\documentclass[10pt]{amsart}
\usepackage{amssymb,amsthm,amsmath,amsfonts}
\usepackage[numbers]{natbib}
\usepackage{hyperref}
\usepackage{enumerate}

\usepackage[normalem]{ulem}

\makeatletter
\g@addto@macro\th@plain{\thm@headpunct{}}
\makeatother

\numberwithin{equation}{section}
\theoremstyle{plain}
\newtheorem{thm}{Theorem}[section]
\newtheorem{lemma}{Lemma}[section]
\newtheorem{remark}{Remark}[section]

\newtheorem{example}{Example}[section]

\def \E {{\mathbb E}}
\def \P {{\mathbb P}}
\def \RR {{\mathbb R}}
\def \eps {\varepsilon}

\def \lf {\underline{f}}
\def \lk {\underline{k}}

\def \uf {\overline{f}}
\def \uk {\overline{k}}
\newcommand{\dd}{\mathrm{d}}

\title[On perpetuities with light tails]{On perpetuities with light tails}
\author[B. Ko\l{}odziejek]{Bartosz Ko\l{}odziejek}
\address{Faculty of Mathematics and Information Science\\Warsaw University of Technology\\Koszykowa 75\\00-662 Warsaw, Poland}
\email{b.kolodziejek@mini.pw.edu.pl}

\keywords{perpetuity; dependence structure; regular variation; {T}auberian theorems; convex conjugate.}

\subjclass[2010]{Primary 60H25; Secondary 60E99}

\begin{document}
\begin{abstract}
In the paper we consider the asymptotics of logarithmic tails of a perpetuity 
$$R \stackrel{d}{=}\sum_{j=1}^\infty Q_j \prod_{k=1}^{j-1}M_k,\qquad(M_n,Q_n)_{n=1}^\infty \mbox{ are i.i.d. copies of }(M,Q),$$ 
in the case when $\P(M\in[0,1))=1$ and $Q$ has all exponential moments. 
If $M$ and $Q$ are independent, under regular variation assumptions, we find the precise asymptotics of $-\log\P(R>x)$ as $x\to\infty$. Moreover, we deal with the case of dependent $M$ and $Q$ and give asymptotic bounds for $-\log\P(R>x)$. It turns out that dependence structure between $M$ and $Q$ has a significant impact on the asymptotic rate of logarithmic tails of $R$. Such phenomenon is not observed in the case of heavy-tailed perpetuities.
\end{abstract}

\maketitle

\section{Introduction}
In the present paper, we consider a random variable $R$ defined as a solution of the affine stochastic equation
\begin{align}\label{defperp} 
R \stackrel{d} {=}MR +Q\qquad R\mbox{ and
}(M,Q)\mbox{ independent.}
\end{align}
Under suitable assumptions (see \eqref{cond} below) on $(M,Q)$, one can
think of $R$ as a limit in distribution of the following iterative scheme:
\begin{align}
\label{it} R_n = M_nR_{n-1}+Q_n,\qquad
n \geq1,
\end{align}
where $(M_n,Q_n)_{n\geq 1}$ are i.i.d. copies of $(M,Q)$ and $R_0$ is arbitrary and independent of $(M_n,Q_n)_{n\geq 1}$. Writing out the
above recurrence and renumbering the random variables $(M_n,Q_n)$, we see that $R$ may also be defined by
\begin{align}\label{perpd} 
R \stackrel{d}{=}\sum_{j=1}^\infty Q_j \prod_{k=1}^{j-1}M_k,
\end{align}
provided that the series above converges in distribution. 
For a detailed discussion of sufficient and necessary conditions in one-dimensional
case, we refer to \citet{ver79} and \citet
{GM00}; here we only note that conditions
\begin{align}
\label{cond} {\mathbb E}\log^+ |Q| < \infty\quad\mbox{and}\quad
{\mathbb E}\log|M| < 0
\end{align}
suffice for the almost sure convergence of the series in \eqref{perpd} and for uniqueness of a solution to \eqref{defperp}. For a systematic approach to the probabilistic properties of the fixed point equation \eqref{defperp} and much more we recommend two recent books \cite{BurBook16} and \cite{IksBook16}.

When $R$ is the solution of \eqref{defperp}, then following a custom
from insurance mathematics, we call $R$ a \emph{perpetuity}. In this
scheme,  let $Q$ represent a random payment and $M$ a random discount factor. Then $R$ is the present value of
a commitment to pay the value of $Q$ every year in the future; see \eqref{perpd}. Such a stochastic equation appears in many areas
of applied mathematics; for a broad list of references consult, for
example, \cite{DF99} and \cite{ver79}.
If $(R,M,Q)$ satisfy \eqref{defperp} we will say that perpetuity $R$ is \emph{generated} by $(M,Q)$ and that random vector $(M,Q)$ is the \emph{generator} of $R$.

For the sake of simplicity, we consider only the case when 
\begin{align}\label{MQpos}
\P(M\geq0,Q\geq0)=1
\end{align}
which implies $\P(R\geq0)=1$.

The main focus of research on perpetuities is their tail behavior. 
Assume for a moment that $Q=1$ a.s. Then, {for $x\geq 1$}, on the set 
$$\left\{M_1>1-\frac{1}{x},\ldots,M_{\left\lfloor x\right\rfloor}>1-\frac{1}{x}\right\}$$
we have 
$$R\geq\sum_{k=1}^{\left\lfloor x\right\rfloor+1}M_1\cdot\ldots\cdot M_{k-1}\geq \sum_{k=1}^{\left\lfloor x\right\rfloor+1}(1-1/x)^{k-1}> (1-e^{-1})x,$$
which gives a lower bound for the tails $\P(R>(1-e^{-1})x)$ of the form 
$$ \P\left(M_1>1-\frac{1}{x},\ldots,M_{\left\lfloor x\right\rfloor}>1-\frac{1}{x}\right)=\P\left(M>1-\frac{1}{x}\right)^{\left\lfloor x\right\rfloor}.$$
It turns out that such approach, proposed in \cite{GG96}, gives the appropriate logarithmic asymptotics for constant $Q$; in \cite{BK17} (with earlier contribution by \citet{HitWes09})
it is proven that under some weak assumptions on the distribution of $M$ near $1-$, one has
\begin{align}\label{TRTRT}
\log\P(R>x)\sim c\,x\log\P\left(M>1-\frac{1}{x}\right)
\end{align}
for an {explicitly} given positive constant $c$. As usual, we write $f(x)\sim g(x)$ if $\lim_{x\to\infty}f(x)/g(x) = 1$. 

The next step in \cite{GG96} was to consider non-constant $Q$. If $Q$ and $M$ are independent, and $M$ has a distribution equivalent at $1$ to uniform distribution, that is,
$$-\log\P\left(M>1-\frac{1}{x}\right)\sim \log x,$$ 
then (see \cite[Theorem 3.1]{GG96})
$$\lim_{x\to\infty}\frac{\log\P(R>x)}{x\log\P\left(M>1-1/x\right)}=\frac{1}{q_+},$$
where $q_+=\mathrm{ess}\sup Q\in(0,\infty]$.

Two natural questions then arise: 
\begin{enumerate}
\item what is the precise asymptotic if $q_+=\infty$?
\item what is the asymptotic if $M$ and $Q$ are not independent?
\end{enumerate}
This paper is devoted to answering both these questions in a unified manner.
We will be particularly interested in the asymptotic behaviour of $\log\P(R>x)$ as $x\to\infty$, which is closely related to the asymptotic behaviour of $\log M_R(t)$, where $M_R$ is the moment generating function of $R$. It is known that if $\P(M>1)>0$, then $R$ is necessarily heavy tailed.
In the present paper we are interested in the case when $\P(M\in[0,1])=1$ and when 
\begin{align}\label{QMGF}
M_Q(t)=\E e^{t Q}<\infty \qquad\mbox{ for all }t\in\RR.
\end{align}
In such case, $R$ is always light-tailed; by \cite{AIR09,BDIM}
$$M_R(t)=\E e^{t R}\mbox{ is finite on the set }(-\infty,t_0),$$
where $t_0:=\sup\{t\colon\E e^{t Q}I_{M=1} <1\}$, which is positive since $\P(M=1)<1$. 
If $t_0$ is finite, then by \cite[Lemma 5]{DG06},
\begin{align}\label{liminf}
\liminf_{x\to\infty} \frac{-\log\P(R>x)}{x}=\sup\{r>0\colon M_R(t)<\infty\}=t_0,
\end{align}
which means that this case is completely solved.
We have $t_0=\infty$ if and only if {either} $\P(M=1)=0$ or $\P_{Q|M=1}=\delta_0$, but the second case can be reduced to the first one. To see this, assume that $\P(M=1)>0$, $\P(Q=0|M=1)=1$ and define $N=\inf\{n\colon M_n<1\}$. It is easy to see that $N$ is a stopping time with respect to $\mathcal{F}_n:=\sigma\left( (M_k,Q_k)\colon\,k\leq n\right)$ and $\P(N<\infty)=1$. Then, the distribution of
$$\left(M_1\cdot\ldots\cdot M_{N-1},\sum_{k=1}^N M_1\ldots M_{k-1}Q_k\right)$$
is the same as the conditional distribution of $(M,Q)$ given $\{M<1\}$. Thus, if $(M^\prime,Q^\prime)\stackrel{d}{=}(M,Q)|M<1$, by \cite[Lemma 1.2]{ver79}, we have
$$R\stackrel{d}{=}M^\prime R+Q^\prime,\qquad R\mbox{ and
}(M^\prime, Q^\prime)\mbox{ independent}$$
and we have $\P(M^\prime=1)=0$.
Therefore, to exclude the case of finite $t_0$ we assume
\begin{align}\label{Ml1}
\P(M\in[0,1))=1.
\end{align}

Observe that the case when $M\leq m_+<1$ and $Q\leq q_+<\infty$ a.s. is uninteresting for us, since then $R$ has no tail (actually, $R\leq q_+/(1-m_+)$ a.s.). We will always exclude this case by assuming that 
\begin{align}\label{MQnb}
\frac{Q}{1-M}\qquad\mbox{is not bounded}.
\end{align}

We note here that the structure of dependence between $M$ and $Q$ does not have a significant impact on the tails of heavy-tailed perpetuities. If
\begin{align}\label{deg}
\P(r=Mr+Q)<1\qquad\mbox{ $r\in\mathbb{R}$}
\end{align}
then, in the cases considered in \cite{Gol91, Gre94, Dys15, DK17, Kev16}, the rate of asymptotics of $\P(R>x)$ is not influenced by the dependence structure of $(M,Q)$ (with possible exception in the very special unsolved case of \cite{DK17} if $\E M^\alpha Q^{\alpha-\eta}=\infty$ for all $\eta\in(0,\alpha)$). The problem becomes more complicated if $(M,Q)$ have lighter tails, that is if the moment generating function of $R$ exists in a neighbourhood of $0$ (but not in $\RR$), but still there is a relatively high insensitivity to the dependence structure of the tail of $R$ for given marginals (this is because in such case $Q$ dominates $M$); see e.g. \cite[Theorem 1.3]{BDIM} and \eqref{liminf}. If the moment generating function of $R$ is finite on the whole $\RR$, we will see that the dependence structure may have significant impact on the rate of convergence even for logarithmic tails, what can be observed in the following example (see also Example \ref{EXQQQ}):
\begin{example}
Consider $(M,Q)=(U,U)$ and $(M',Q')=(U,1-U)$, where $U$ is uniformly distributed on $[0,1]$ {(note that \eqref{MQnb} and \eqref{deg} are not satisfied here)}.  Let
$R$ and $R'$ be the perpetuities generated by $(M,Q)$ and  $(M',Q')$, respectively.
We have
$$-\log\P(R>x)\sim x\log x,$$
while
$\P(R'=1)=1$. 
To see the first result, observe that $\tilde{R}=R+1$ satisfies 
$$\tilde{R}\stackrel{d}{=}U\tilde{R}+1,\qquad \tilde{R}\mbox{ and }U\mbox{ are independent},$$
thus the results of \cite{GG96} and \cite{BK17} apply.
In this very example, {asymptotics of $\P(R>x)$ as $x\to\infty$} is also known \cite{ver72}.
\end{example}
Finally, we would like to mention here \cite{TY16}, where the authors considered generators fulfilling a certain dependence structure which somehow resembles the notion of asymptotic independence from \cite{MRR02}. A similar and significantly weaker, but still restrictive condition was considered in \cite[Eq. (5)]{BDIM}. Here we will be able to give bounds for the logarithmic tails even if large values of $M$ exclude large values of $Q$ (and vice versa), which is in opposition to the asymptotic independence. 

The paper is organized as follows. 
\begin{description}
\item[Section \ref{secPreli}: Preliminaries] In the next section we give a short introduction to the theories that will be extensively exploited, that is, regular variation, convex analysis, {T}auberian theorems and concepts of dependence. 
\item[Section \ref{Sind}: Independent $M$ and $Q$] 
We find precise asymptotics of logarithmic tail of $R$ when $M$ and $Q$ are independent and $Q$ is unbounded (Theorem \ref{thmIND1}), and bounded (Theorem \ref{thmIND2}). Particularly, we assume that
\begin{align*}
x&\mapsto-\log\P\left(\frac{1}{1-M}>x\right)\in\mathcal{R}_{r-1},\quad r>1,
\intertext{and}
x&\mapsto-\log\P(Q>x)\in\mathcal{R}_{\alpha},\quad \alpha>1,\qquad \mbox{or}\qquad \P(Q\leq q_+)=1,
\end{align*}
where $\mathcal{R}_\gamma$ denotes the class of regularly varying functions with index $\gamma$. Under these assumptions \eqref{QMGF}, \eqref{Ml1} and \eqref{MQnb} are satisfied.

We show that
$$-\log\P(R>x)\sim c\,h(x),$$
where constant $c>0$ is given explicitly and
\begin{align}\label{defh}
h(x):=\inf_{t\geq1}\left\{ -t\log\P\left(\frac{1}{1-M}>t, Q>\frac xt\right)\right\}.
\end{align}
Observe that if $Q=1$ a.s., then $h(x)=-x\log\P(M>1-1/x)$, so we recover \eqref{TRTRT}.
Thus, we generalize the results of \cite{GG96} and \cite{BK17}, but with new proofs, which are very different from the ones in \cite{GG96,BK17}. Our proofs are based on a new formulation of the classical {T}auberian theorems; see Section \ref{TAUB}.

The appearance of function $h$ is probably the most interesting phenomenon here. {It should be noted here that the function $h$ (in the simple form when $Q$ is degenerate) in the two sided bounds for $\log\P(R>x)$ appeared for the first time in \cite{Hit10}.}
\item[Section \ref{Sh}: Properties of function $h$] This {s}ection is devoted to explaining some informal heuristics, which show that the function $h$ is a natural candidate for describing the asymptotic of $-\log\P(R>x)$ when $M$ and $Q$ are not independent. 
In Theorem \eqref{propH} we give basic properties of the function $h$.

\item[Section \ref{secLower}: Lower bound] In this {s}ection we find lower bound for the quotient of $\log\P(R>x)/h(x)$ as $x\to\infty$. To present our argument shortly, let us consider scalar sequences $(\delta_k)_{k\geq 1}$ and $(q_k)_{k\geq1}$ such that
\begin{align}\label{strcnd}
x{\leq}\sum_{k=1}^n(1-\delta_1)\cdot\ldots\cdot(1-\delta_{k-1}) q_k.
\end{align}
Then, \eqref{perpd} implies
$$\log\P(R>x)\geq \log\P\left(\{R>x\}{\cap_{k=1}^n} \{ M_k>1-\delta_k, Q_k>q_k\}\right) = \sum_{k=1}^n\log\P(M>1-\delta_k,Q>q_k)$$
and so
$$
\liminf_{x\to\infty}
\frac{\log\P(R>x)}{h(x)} \geq
\liminf_{x\to\infty} \sup_{(\delta_k,q_k)_k, n}\left\{ \sum_{k=1}^n \frac{\log\P(M>1-\delta_k,Q>q_k)}{h(x)}\,\colon\,\, \mbox{\eqref{strcnd} holds}\right\},
$$
{where the supremum is taken over all possible choices of $n\in\mathbb{N}$ and sequences $(\delta_k)_{k=1}^n$, $(q_k)_{k=1}^n$ for which \eqref{strcnd} holds.}
In Theorem \ref{lower} we were able to calculate the right hand side of the above inequality under some regularity assumptions on $h$.
\item[Section \ref{secUpper}: Upper bound]
We show that if $R$ is generated by $(M,Q)$ with an arbitrary dependence structure, then
$$\limsup_{x\to\infty} \frac{\log\P(R>x)}{h_{co}(x)}\leq \lim_{x\to\infty} \frac{\log\P(R_{co}>x)}{h_{co}(x)}=-c,$$ 
where $R_{co}$ is a perpetuity generated by the so-called comonotonic $(M,Q)$ (see Section \ref{secDep}) and $h_{co}$ is the corresponding function $h$. Constant $c$ is given explicitly (see Theorem \ref{THco}). 
In this section we give stronger results under additional assumptions that the vector $(M,Q)$ is positively or negatively quadrant dependent (Theorem \ref{NQD}).
\item[Section \ref{secProof}: Proofs] The last section contains proofs of some results from preceding {s}ections.
\end{description}

\section{Preliminaries}\label{secPreli}
\subsection{Regular variation}
In this section we give a brief introduction to the theory of regular variation. For further details we refer to Bingham et al. \citep{BGT89}.

A positive measurable function $L$ defined in a neighborhood of $+\infty$ is said to be \emph{slowly varying} if 
\begin{align}\label{reg}
\lim_{x\to\infty}\frac{L(t x)}{L(x)}=1,\qquad\mbox{for all }t>0.
\end{align}
A positive measurable function $f$ defined in a neighborhood of $+\infty$ is said to be \emph{regularly varying} with \emph{index} $\rho\in\RR$ if $f(x)=x^\rho L(x)$ with $L$ slowly varying. We denote the class of regularly varying function with index $\rho$ by $\mathcal{R}_\rho$, so that, $\mathcal{R}_0$ is the class of slowly varying functions.


We say that a positive function $f$ \emph{varies smoothly with index $\rho$} ($f\in \mathcal{SR}_\rho$), if $f\in C^\infty$ and for all $n\in\mathbb{N}$,
\begin{align}\label{SR}
\lim_{x\to\infty}\frac{x^n f^{(n)}(x)}{f(x)}=\rho(\rho-1)\ldots(\rho-n+1).
\end{align}
It is clear that $\mathcal{SR}_\rho\subset \mathcal{R}_\rho$. Moreover, if $f\in \mathcal{SR}_\rho$ then $x^2 f^{\prime\prime}(x)/f(x)\to \rho(\rho-1)$, hence $f$ is ultimately strictly convex if $\rho>1$; \emph{ultimately} here and later means ``{on the vicinity of infinity}''. Further, if $f\in \mathcal{SR}_\rho$ with $\rho>0$, then on the neighbourhood of infinity $f$ has an inverse in $\mathcal{SR}_{1/\rho}$ (\cite[Theorem 1.8.5]{BGT89}).
For any $f\in \mathcal{R}_\rho$ there exist $\lf, \uf\in \mathcal{SR}_\rho$ with $\lf(x)\sim \uf(x)$ and $\lf\leq f\leq \uf$ on a neighbourhood of infinity (the Smooth Variation Theorem \cite[Theorem 1.8.2]{BGT89}).

If $f\in\mathcal{SR}_\gamma$ with $\gamma>0$, then
\begin{align}\label{Rg}
\lim_{x\to\infty}\frac{f\left(x+u\frac{f(x)}{f^\prime(x)}\right)}{f(x)}= \left(1+\frac{u}{\gamma}\right)^\gamma.
\end{align}
This follows by the fact that convergence in \eqref{reg} and \eqref{SR} is locally uniform, see \cite[Theorem 1.2.1]{BGT89}.
We say that a measurable function $f$ is \emph{rapidly varying} ($f\in \mathcal{R}_\infty$) if 
$$\lim_{x\to\infty} \frac{f(t x)}{f(x)}=\infty,\qquad\mbox{for all $t>1$}.$$
It is the subclass of $\mathcal{R}_\infty$ that we are interested in. The class $\Gamma$ consists of nondecreasing and right-continuous
functions $f$ for which there exists a measurable function $g\colon
{\mathbb R}\to(0,\infty)$ such that (see \cite{BGT89}, Section~3.10)
\begin{align}\label{Gamma} 
\lim_{x\to\infty}\frac{f (x+ug(x) )}{f(x)}= e^u,\qquad \mbox{for all $u\in\mathbb{R}$.}
\end{align}
Function $g$ in \eqref{Gamma} is called \emph{an auxiliary function}
and if $f$ has nondecreasing positive derivative, then one may take
$g=f/f^\prime$ (compare with \eqref{Rg}). 
It can be shown that if $f\in\Gamma$ and $t>1$, then
$\lim_{x\to\infty} f(t x)/f(x)=\infty$, thus $\Gamma\subset \mathcal{R}_\infty$.

The class $\Gamma$ is very rich: If $f_1\in \mathcal{R}_\rho$, $\rho>0$ and
$f_2\in\Gamma$, then $f_1\circ f_2\in\Gamma$ (\cite{BGT89},
Proposition~3.10.12). The same holds if $f_1\in\Gamma$ and $f_2'\in
\mathcal{R}_\rho$ with $\rho>-1$ or if $f_1, f_2^\prime\in\Gamma$ (\cite{BGT89},
page~191). 

Finally, we note that convergence in \eqref{Rg} is uniform on compact subsets of $(-\gamma,\infty)$ and
that convergence in \eqref{Gamma} is uniform on compact subsets of $\RR$ (see \cite[Proposition~3.10.2]{BGT89}). 

\subsection{Convex conjugate}\label{CC}

For a function $f\colon (0,\infty)\to\RR$ we define its \emph{convex conjugate} (or the Fenchel-Legendre transform) by
\begin{align}\label{conv}
f^\ast(x)=\sup\{xz-f(z)\colon z>0\}.
\end{align}
It is standard that $f^\ast$ is convex, non-decreasing and lower semi-continuous. Moreover, if $f$ is convex and lower semi-continuous then $(f^{\ast})^\ast=f$ (\citep{Roc70}). 
Convex-conjugacy is order-reversing, that is, if $f\leq g$, then $f^\ast\geq g^\ast$.

If $f$ is differentiable and strictly convex, then the supremum \eqref{conv} is attained at $z=(f^\prime)^{-1}(x)$ and thus 
$f^\ast(x)=x(f^\prime)^{-1}(x)-f((f^\prime)^{-1}(x)).$
Moreover, $f^\prime\circ (f^\ast)^\prime=(f^\ast)^\prime\circ f^\prime=\mathrm{Id}$ and so
\begin{align}\label{eqstar}
f^\ast(x)=x(f^\ast)^\prime(x)-f\left((f^\ast)^\prime(x)\right).
\end{align}

We will be interested in the relation between $f$ and $f^\ast$ when $f$ is regularly varying. We say that $\alpha$ and $\beta$ are \emph{conjugate numbers }if $\alpha,\beta>1$ and 
\mbox{$\alpha^{-1}+\beta^{-1}=1$}. 
Let $L$ be a slowly varying function. 
Then (\cite[Theorem 1.8.10, Corollary 1.8.11]{BGT89})
$$f(x)\sim \frac{1}{\alpha} x^\alpha L(x^\alpha)^{1/\beta}\in \mathcal{R}_\alpha$$
if and only if
$$f^\ast(z)\sim \frac{1}{\beta} z^\beta L^\#(z^\beta)^{1/\alpha}\in \mathcal{R}_\beta,$$
where $L^\#$ is a dual, unique up to asymptotic equivalence, slowly varying function with
$$L(x)L^\#\left(x L(x)\right)\to 1,\qquad L^\#(x)L\left(xL^\#(x)\right)\to 1,\qquad \mbox{ as }x\to\infty.$$

By the very definition of $f^\ast$ we obtain Young's inequality 
$$f(s)+f^\ast(t)\geq st\qquad\mbox{for all $s,t>0$}.$$
If $f$ and $f^\ast$ are invertible, then taking $s=f^{-1}(x)$ and $t=(f^\ast)^{-1}(x)$ for $x>0$, we have 
$$\frac{(f^\ast)^{-1}(x)f^{-1}(x)}{x}\leq 2.$$
We will show that left hand side above has a limit as $x\to\infty$.
If $f\in\mathcal{R}_\rho$ with $\rho>0$, then there exists a function $g$, such that $f(g(x))\sim g(f(x))\sim x$. Such $g$ is unique up to asymptotic equivalence {(see \cite[Theorem 1.5.12]{BGT89})} and is called \emph{asymptotic inverse} of $f$. If $f$ is locally bounded on $(0,\infty)$, then one can take $g=f^{\leftarrow}$, where
$$f^{\leftarrow}(x)=\inf\{y\in(0,\infty)\colon f(y)>x\}.$$
\begin{lemma}\label{kinv}
	Let $f\in \mathcal{R}_\alpha$ with $\alpha>1$ and let $\beta$ be a conjugate number to $\alpha$. Then
	$$\frac{f^{\leftarrow}(x)(f^\ast)^{\leftarrow}(x)}{x}\to \alpha(\beta-1)^{1/\beta}\mbox{ as }x\to\infty.$$
\end{lemma}
The proof is postponed to the last {s}ection.

The following theorem will be important for us. For a formulation in $\RR^n$ see \cite[Theorem 2.5.1]{HU93}.
\begin{thm}\label{HU}
Assume that functions $a$ and $b$ are lower semi-continuous and convex on $(0,\infty)$. If $a$ is additionally non-decreasing, then for $x>0$ one has
$$(a\circ b)^\ast(x)=\inf_{z>0}\left\{a^\ast(z)+z\,b^\ast\left(\frac xz\right)\right\}.$$
\end{thm}	 

\subsection{Dependence structure of random vectors}\label{secDep}
A function $f\colon \RR^2\to\RR$ is said to be \emph{supermodular} if
$$f(\min\{u,v\})+f(\max\{u,v\})\geq f(u)+f(v),\qquad\mbox{ for all }u,v\in \RR^2,$$
where the minimum and maximum are calculated component-wise.
If $f$ has continuous second order partial derivatives, then $f$ is supermodular if and only if $\frac{\partial^2 f}{\partial x\partial y}\geq 0$. One of the important examples of supermodular functions is $f(x_1,x_2)=g(x_1+x_2)$, when $g$ is convex. We will use this fact in the proof of Lemma \ref{LB} below.

{A random vector $(X,Y)$ is said to be smaller than a random vector $(X',Y')$ in the \emph{supermodular order}
if $\E f(X,Y)\leq \E f(X',Y')$ for all supermodular functions $f$ for which the expectations exist.}
The following theorem has many formulations with different assumptions (see {e.g.} \cite{Lor53, Tch80}), but we will use the one given by \citet{Cam76}.
\begin{thm}\label{Camb}
Let $f\colon \RR^2\to\RR$  be a continuous supermodular function. Let $(X,Y)$ and $(X',Y')$ be random vectors with the same marginal distributions. Assume that 
$$\P(X\leq x, Y\leq y)\leq \P(X'\leq x, Y'\leq y),\qquad\mbox{ for all }x,y\in\RR.$$
If the expectations $\E f(X,y_0)$ and $\E f(x_0,Y)$ are finite for some $x_0$ and $y_0$, then
$$\E f(X,Y)\leq \E f(X',Y')$$
provided that the above expectations exist (even if infinite valued).
\end{thm}
Assume that $X$ and $Y$ are random variables defined on the same probability space.
Let $F_X$ and $F_Y$ denote the cumulative distribution function (CDF) of $X$ and $Y$, respectively. Define $\underline{F}(x,y)=(F_X(x)+F_Y(y)-1)_+$ and $\overline{F}(x,y)=\min\{F_X(x),F_Y(y)\}$. It is clear that $\underline{F}$ and $\overline{F}$ are two dimensional cumulative distribution functions. 
Moreover, $\underline{F}$ and $\overline{F}$ have the same marginal distributions and for any $F$ with the same marginals one has (Fr\'{e}chet–-Hoeffding bounds)
$$\underline{F}\leq F\leq\overline{F}.$$
{If a random variable or vector $X$ has a CDF $F$, we will write $X\stackrel{d}{\sim}F$.}
We say that a vector $(X,Y){\stackrel{d}{\sim}} F$ is \emph{comonotonic} if $F=\overline{F}$ and that it is \emph{countermonotonic} if $F=\underline{F}$.
Thus, Theorem \ref{Camb} implies that comonotonic (countermonotonic) random vectors are maximal (minimal) with respect to the supermodular order.
For a CDF $F$ define for $x\in[0,1]$,
$$F^{-1}(x)=\inf\{y\in\mathbb{R}\colon F(y)\geq x\}.$$
It is known that if $U$ is uniformly distributed on $[0,1]$, then
$$(F_X^{-1}(U),F_Y^{-1}(U)){\stackrel{d}{\sim}} \overline{F}$$
and
$$(F_X^{-1}(U),F_Y^{-1}(1-U)){\stackrel{d}{\sim}} \underline{F}.$$

We say that the pair $(X,Y)$ is \emph{positively quadrant dependent} (\cite{L66, KS89}) if 
$$\P(X\leq x, Y\leq y)\geq\P(X\leq x)\P(Y\leq y)\qquad\mbox{for all $x,y\in\RR$}.$$
Similarly, $(X,Y)$ is \emph{negatively quadrant dependent} if above holds with the inequality sign reversed.
{We say that a function $f$ is weakly monotonic if it is non-decreasing or on-increasing.}
\begin{lemma}\label{LB}
	Assume \eqref{MQpos} and let $(M',Q')$ be a random vector such that
	$$\P(M\leq x, Q\leq y)\leq \P(M'\leq x, Q'\leq y),\qquad\mbox{ for all }x,y\in\RR$$
	with $M'\stackrel{d}{=}M$ and $Q'\stackrel{d}{=}Q$. 
	Let $R$ and $R'$ denote the perpetuities generated by $(M,Q)$ and $(M',Q')$, respectively.
	Then, 
	\begin{align}\label{convf}
		\E f(R)\leq\E f({R'})\qquad\mbox{for all convex and {weakly monotonic} functions $f$ on $\RR$},
	\end{align}
	provided that the above expectations exist (even if infinite valued).
\end{lemma}
Proof of Lemma~\ref{LB} is postponed to Section \ref{secProof}.

\begin{remark}
Assume additionally that 
$$\E M<1\qquad\mbox{ and }\qquad\E Q<\infty.$$
In such case $\E R$ and $\E R'$ are finite and 
\begin{align}\label{ERp}
\E R=\frac{\E Q}{1-\E M}=\frac{\E Q'}{1-\E M'}=\E R'.
\end{align}

For convex and non-decreasing $f_{x}(r)=(r-x)_+$ with $x>0$ we have
$$\E f_{x}(R)=\int_{x}^\infty \P(R>t)\dd t$$
and thus \eqref{convf} gives us
$$\int_{x}^\infty \left(\P(R'>t)-\P(R>t)\right)\dd t\geq 0,\qquad\mbox{for all $x$}.$$
But, by \eqref{ERp} we obtain
 $$\int_{0}^\infty \left(\P(R'>t)-\P(R>t)\right)\dd t=\E(R'-R)=0,$$
 which implies that
$$\int_{-\infty}^{x}\left({F_{R'}(t)-F_{R}(t)} \right)\geq 0\qquad\mbox{for all $x$},$$ 
{which is equivalent to saying that}
	{$R$} is second-order stochastically dominant over {$R'$}; see \cite{RS70}.
\end{remark}

\subsection{Useful Tauberian theorems}\label{TAUB}
Tauberian Theorems presented below are classical, but here we formulate them in a new way. To see that these formulations are equivalent to classical ones, see {Section~\ref{secProof}}.

\begin{thm}[Kasahara's Tauberian Theorem]\label{Kas}
	Let $X$ be a a.s. non-negative random variable such that the moment generating function
	$$M(z)=\E e^{z X}$$
	is finite for all $z>0$.
	Let $k\in \mathcal{R}_\rho$ with $\rho>1$. Then,	
	$$-\log\P(X>x)\sim k(x)$$
	if and only if 
	$$\log M(z)\sim k^\ast(z).$$
Moreover, we have the following result on limits of oscillation:
\begin{align*}
B_1\leq\liminf_{x\to\infty}\frac{-\log\P(X>x)}{k(x)}\leq \limsup_{x\to\infty}\frac{-\log\P(X>x)}{k(x)}\leq B_2
\end{align*}
for some constants $0<B_1<B_2<\infty$ if and only if
\begin{align*}
\tilde{B}_1\leq\liminf_{z\to\infty}\frac{\log M(z)}{f^\ast(z)}\leq \limsup_{z\to\infty}\frac{\log M(z)}{f^\ast(z)}\leq \tilde{B}_2.
\end{align*}
for some constants $0<\tilde{B}_1<\tilde{B}_2<\infty$ (above result can be strengthened by specifying the relation between $B_i$ and $\tilde{B}_i$; see Corollary~4.12.8 \cite{BGT89}).
\end{thm}

\begin{thm}[de Bruijn's Tauberian Theorem]\label{deB}
	Let $Y$ be a non-negative random variable. Let $f\in \mathcal{R}_\rho$ with $\rho>1$.
	Then
	$$-x\log\P\left(Y< \frac{1}{x}\right)\sim f(x) \mbox{ as }x\to\infty$$
	if and only if
	$$-\log \E e^{-\lambda Y} \sim (f^\ast)^{\leftarrow}(\lambda) \mbox{ as }\lambda\to\infty.$$
\end{thm}

\section{Independent generators}\label{Sind}
In the following {s}ection we consider $M$ and $Q$ independent under two regimes:
\begin{itemize}
\item both $1/(1-M)$ and $Q$ are unbounded - Theorem \ref{thmIND1},
\item $1/(1-M)$ is unbounded, while $Q$ is bounded - Theorem \ref{thmIND2}.
\end{itemize}
Both of the proofs use two {T}auberian theorems introduced in the previous {s}ection.

\begin{thm}\label{thmIND1}
Let $M$ and $Q$ be independent and assume \eqref{MQpos}.
Let 
$$k(x):=-\log\P(Q>x)\quad\mbox{ and }\quad f(x):=-x\log\P(M> 1-1/x)$$
and assume that $f\in\mathcal{R}_r$ and $k\in\mathcal{R}_\alpha$ with $r,\alpha>1$. Let $r^\ast$ and $\beta$ denote the conjugate  numbers to $r$ and $\alpha$, respectively.
Then $(f^\ast\circ k^\ast)^\ast\in \mathcal{R}_\gamma$ and
\begin{align}\label{indres}
-\log\P(R>x)\sim \left(\frac{\gamma}{\gamma-1}\right)^{\gamma-1}(f^\ast\circ k^\ast)^\ast(x)
\end{align}
with $\gamma=\beta r^\ast/(\beta r^\ast-1)$.
\end{thm}
{As will be seen in Remark \ref{hfk} and Theorem \ref{propH}, function $(f^\ast\circ k^\ast)^\ast$ coincides with function $h$ introduced in \eqref{defh}.}

Similarly, we can handle the case of bounded $Q$.
\begin{thm}\label{thmIND2}
Let $M$ and $Q$ be independent and assume \eqref{MQpos}. 
Let 
$$q_+:=\mathrm{ess}\sup Q<\infty\quad\mbox{ and }\quad f(x):=-x\log\P(M> 1-1/x) $$
and assume that $f\in\mathcal{R}_r$ with $r>1$. Then,
\begin{align*}
-\log\P(R>x)\sim \left(\frac{r}{r-1}\right)^{r-1}f\left(\frac{x}{q_+}\right).
\end{align*}
\end{thm}

\begin{proof}[Proof of Theorem~\ref{thmIND1}]
Since $M$, $Q$ and $R$ are independent on the right hand side of $R\stackrel{d}{=}MR+Q$, for 
$$\psi(z):=\log \E e^{z R}$$
we have
\begin{align}\label{eqind0}
e^{\psi(z)}=\E e^{zM R}\E e^{zQ}=\E e^{\psi(zM)}\E e^{zQ}
\end{align}
upon conditioning on $M$.

In view of Kasahara's Tauberian Theorem \ref{Kas}, it is enough to show that
\begin{align}\label{defpsi}
	\psi(z)\sim (\beta r^\ast)^{-1} (f^\ast\circ k^\ast)(z).
\end{align}
Indeed, observe that in such case
\begin{align}\label{nn}
-\log\P(R>x)\sim\psi^\ast(x)\sim\sup_{z>0}\{ z x-(\beta r^\ast)^{-1} (f^\ast\circ k^\ast)(z)\} = (\beta r^\ast)^{-1} (f^\ast\circ k^\ast)^\ast\left(\beta r^\ast x\right).
\end{align}
Since $f^\ast\circ k^\ast\in\mathcal{R}_{\beta r^\ast}$, \eqref{indres} then follows by regular variation of $(f^\ast\circ k^\ast)^\ast\in \mathcal{R}_\gamma$.

Moreover, by the Abelian (direct) parts of the Kasahara's and de Bruijn's Tauberian Theorems (put $X=Q$ and $Y=1-M$) we have
\begin{align*}
\log\E e^{zQ}&\sim k^\ast(z)\in\mathcal{R}_\beta
\intertext{and}
-\log\E e^{-(1-M)z}&\sim (f^\ast)^{\leftarrow}(z)\in\mathcal{R}_{1/r^\ast}.
\end{align*}

Assume for a while that
\begin{align}\label{eqind1}
\log \E e^{z Q}\sim -\log \E e^{-z \psi^\prime(z)(1-M)}.
\end{align}
Then, by the above considerations we obtain
$$k^\ast(z)\sim (f^\ast)^{\leftarrow}(z \psi^\prime(z))$$
or equivalently, {(recall the definition of asymptotic inverse in Section \ref{CC})}
$$(f^\ast\circ k^\ast)(z)\sim z \psi^\prime(z).$$
This implies that $\psi^\prime\in\mathcal{R}_{\beta r^\ast-1}$ and so $z\psi^\prime(z)\sim\beta r^\ast\psi(z)$,
which, together with the above equation, gives \eqref{defpsi} after applying Kasahara's Tauberian Theorem {(see \eqref{nn})}.

It is left to show that \eqref{eqind1} holds.
By convexity of $\psi$, we have
\begin{align}\label{psiconv1}
\E e^{\psi(zM)-\psi(z)}\geq \E e^{-z\psi^\prime(z)(1-M)}.
\end{align}
Moreover, since $R$ is a.s. non-negative, $\psi$ is non-decreasing. Thus, for any $m\in(0,1)$ by monotonicity and again by convexity of $\psi$, we obtain
\begin{align*}
\E e^{\psi(zM)-\psi(z)} &\leq \E e^{-z\psi^\prime(zM)(1-M)}I_{M>m}+e^{\psi(zm)-\psi(z)}\P(M\leq m) =:I_1+I_2.
\end{align*}
Since $\psi$ is strictly convex, we have
\begin{align*}I_1&\leq \E e^{-z\psi^\prime(zm)(1-M)}I_{M>m}\leq \E e^{-z\psi^\prime(zm)(1-M)} 
\intertext{and}
I_2&\leq e^{-z\psi^\prime(zm)(1-m)}.
\end{align*}
But
$$\frac{\E e^{-z\psi^\prime(zm)(1-M)}}{e^{-z\psi^\prime(zm)(1-m)}}= \E e^{-z\psi^\prime(zm)(m-M)}\to\infty\qquad(z\to\infty)$$
as $\P(M>m)>0$, hence
$$\E e^{\psi(zM)-\psi(z)}\leq I_1+I_2\leq \E e^{-z\psi^\prime(zm)(1-M)}(1+o(1))\leq \E e^{-mz\psi^\prime(zm)(1-M)}(1+o(1)),$$
because $m<1$. Thus, by \eqref{eqind0} we obtain that
$$\log\E e^{\tfrac zm Q}=-\log\E e^{\psi\left(\tfrac zm M\right)-\psi\left(\tfrac zm\right)}\geq -\log \E e^{-z\psi^\prime(z)(1-M)}-\log(1+o(1))$$
Hence by \eqref{psiconv1} and the above inequality, for any $m\in(0,1)$, we have
$$\log\E e^{z Q} \leq -\log \E e^{-z\psi^\prime(z)(1-M)}\leq\log\E e^{z/mQ}+o(1). $$
By the regular variation of $z\mapsto\log\E e^{zQ}$, we finally conclude that
$$
1\leq\liminf_{z\to\infty}\frac{-\log \E e^{-z\psi^\prime(z)(1-M)}}{\log\E e^{zQ}}\\
\leq\limsup_{z\to\infty}\frac{-\log \E e^{-z\psi^\prime(z)(1-M)}}{\log\E e^{zQ}}\leq m^{-\beta}
$$
for any $m\in(0,1)$, which is \eqref{eqind1}.
\end{proof}

\begin{proof}[Proof of Theorem~\ref{thmIND2}]
The proof proceeds in the same way as previously, but here we will have $z\mapsto \log\E \exp(zQ)\in\mathcal{R}_1$ so that $\beta=1$.  Indeed, for any $q\in(0,q_+)$ we have
$$z q_+\geq \log \E e^{zQ}\geq \log\E e^{zQ}I_{Q>q}\geq z q+\log\P(Q>q),$$
which means that $\log\E \exp(zQ)\sim z q_+$. Let $r^\ast$ {be} the conjugate number to $r$. Similarly as before, we show that
\begin{align*} 
z q_+\sim \log \E e^{zQ} =-\log \E e^{\psi(zM)-\psi(z)}
\sim -\log\E e^{-z \psi^\prime(z)(1-M)}\sim (f^\ast)^{\leftarrow}(z\psi^\prime(z))
\end{align*}
so that
$$z\psi^\prime(z)\sim f^\ast(z q_+)\sim r^\ast \psi(z)$$
since $f^\ast\in\mathcal{R}_{r^\ast}$.
Then, by Kasahara's Tauberian theorem, we conclude that
$$
-\log\P(R>x)\sim \psi^\ast(x)\sim \sup_{z>0}\{ zx-\frac{1}{r^\ast}f^\ast(q_+z)\}=\frac{1}{r^\ast}f\left(r^\ast\frac{x}{q_+}\right).
$$
\end{proof}

\section{Heuristics and function $h$}\label{Sh}
In {this} {s}ection we present some informal heuristics, which show that function $h$ defined in \eqref{defh} is a natural candidate for explaining asymptotic of $-\log\P(R>x)$ even if $M$ and $Q$ are not independent.
By Kasahara's Theorem, we know that $x\mapsto-\log\P(R>x)$ is regularly varying with index $\gamma>1$ if and only if $z\mapsto\psi(z):=\log\E \exp(z R)$ is regularly varying with index $\gamma/(\gamma-1)$, where $\psi$ is uniquely determined by the equation
\begin{align*}
\E e^{zQ+\psi(zM)-\psi(z)}=1.
\end{align*}
In such case, we expect that in some sense as $z\to\infty$ we have
$$\E e^{zQ-\psi(z)\left(1-M^{\gamma/(\gamma-1)}\right)}\approx 1$$
and from this point it is not far to considering a function $\lambda$ defined by the equation
$$\E e^{zQ-\lambda(z)(1-M)}=1\qquad\mbox{for $z>0$}.$$
It seems reasonable to expect that for large $z$ and some constants $B_i$, $i=1,2$, one has (this is true if $m_-=\mathrm{ess}\inf M>0$)
$$0<B_1\leq \frac{\psi(z)}{\lambda(z)}\leq B_2<\infty.$$
Assume now that $\lambda$ is regularly varying. By Kasahara's Tauberian theorem, this would imply that (recall that $-\log\P(R>x)\sim \psi^\ast(x)$)
$$0<\tilde{B}_1\leq \liminf_{x\to\infty} \frac{-\log\P(R>x)}{\lambda^\ast(x)}\leq \limsup_{x\to\infty} \frac{-\log\P(R>x)}{\lambda^\ast(x)}\leq \tilde{B}_2<\infty,$$
for some constants $\tilde{B}_i$, $i=1,2$. However, the definition of $\lambda$ does not seem much more appealing than that of $\psi$, but it is the function $\lambda^\ast$ that is of our interest. 
By the definition of $\lambda$ we have 
$$1=\E e^{zQ-\lambda(z)(1-M)}\geq \E e^{zQ-\lambda(z)(1-M)}I_{1-M<1/t}\geq \E e^{zQ}I_{1-M<1/t} e^{-\lambda(z)/t},$$
which gives for any $t>0$,
\begin{align}\label{ineqlam}
\lambda(z)\geq t\log\E e^{zQ}I_{1-M<1/t}.
\end{align}
Further, by the exponential Markov inequality we have for $z>0$,
$$\P\left(1-M<\frac1t,Q>\frac{x}{t}\right)\leq \frac{\E e^{z Q}I_{1-M<1/t}}{e^{z x/t}},$$
which gives together with \eqref{ineqlam}
$$-t\log\P\left(\frac{1}{1-M}>t,Q>\frac xt\right)\geq zx-t\log \E e^{zQ}I_{M>1-1/t}\geq zx-\lambda(z)$$
for any positive $x$, $t$ and $z$.
Taking $\inf_{t\geq1}$ and $\sup_{z>0}$ of both sides, we obtain (recall the definition of $h$ in \eqref{defh})
$$h(x)\geq \lambda^\ast(x)\qquad \mbox{for all }x>0.$$ 
In general, we are not able to prove that $h(x)\sim\lambda^\ast(x)$ (or $\lim\sup_{x\to\infty} h(x)/\lambda^\ast(x)<\infty$), but there is a strong evidence that such claim is true for a wide class of distributions of $(M,Q)$.
This would eventually imply that 
$-\log\P(R>x)$ is comparable, up to a constant, with $h(x)$ as $x\to\infty$. Moreover, if $M$ and $Q$ are independent, then Theorems \ref{thmIND1} and \ref{thmIND2} give us asymptotics of $-\log\P(R>x)$ in terms of $h$; see below.
\begin{remark}\label{hfk}
Every convex conjugate is convex, non-decreasing and lower semi-continuous. Thus, under assumptions of Theorem \ref{thmIND1}, by Theorem \ref{HU}, we have
\begin{align*}
(f^\ast\circ k^\ast)^\ast(x)=\inf_{t> 0}\left\{ f(t)+t\,k\left(\frac{x}{t}\right)\right\}\sim \inf_{t\geq1}\left\{ -t\log\P\left[\left(M>1-\frac1t\right)\P\left(Q>\frac xt\right)\right]\right\},
\end{align*}
since $f(t)=0$ for $t\in(0,1)$.
Particularly, if $f(x)=c x^r$ and $k(x)=d x^\alpha$ for some $c,d>0$ and $r,\alpha>1$, then direct calculation gives us
		$$
		(f^\ast\circ k^\ast)^\ast(x)=
		d \frac{\alpha+r-1}{r} \left(\frac{c}{d}\frac{r}{\alpha-1}\right)^{\frac{\alpha-1}{\alpha+r-1}}x^{\frac{\alpha r}{\alpha+r-1}}.
		$$
\end{remark}

We gather the properties of function $h$ in the following theorem. Its proof is postponed to the last {s}ection.
\begin{thm}\label{propH}
Assume \eqref{MQpos} and define
$$ f(x):=-x\log\P(M> 1-1/x),\qquad k(x):=-\log\P(Q>x).$$
\begin{itemize}
\item[a)] There exists a function $t$ such that 
\begin{align}\label{XT}
h(x)=-t(x)\log\P\left(\frac{1}{1-M}>t(x), Q>\frac{x}{t(x)}\right)+o(1).
\end{align}
Moreover, if \eqref{MQnb} holds, then
\begin{align}\label{hineq}
t(x)\leq \frac{h(x)+o(1)}{-\log\P\left(\frac{Q}{1-M}> x\right)}.
\end{align}

\item[b)] One has
$$h_{co}\leq h\leq h_{counter},$$
where
\begin{align*}
h_{co}(x):=\inf_{t\geq 1}\left\{ -t\log\min\left\{\P\left(\frac{1}{1-M}>t\right),\P\left(Q>\frac xt\right)\right\}\right\}
\intertext{and}
h_{counter}(x):=\inf_{t\geq 1}\left\{ -t\log\left[\P\left(\frac{1}{1-M}>t\right)+\P\left(Q>\frac xt\right)-1\right]\right\}
\end{align*}
are functions corresponding to co- and countermonotonic vectors $(M,Q)$.

\item[c)] Let 
$$h_{ind}(x):=\inf_{t\geq 1}\left\{f(t)+tk\left(\frac xt\right)\right\}$$
be the $h$ function corresponding to independent $M$ and $Q$. 
Then
$$h_{ind}(x)\sim (f^\ast\circ k^\ast)^\ast(x).$$
If $f\in\mathcal{R}_r$ and $k\in\mathcal{R}_\alpha$ with $r,\alpha>1$, then
$$h_{ind}\in\mathcal{R}_\gamma,$$
where $\gamma=\alpha r/(\alpha+r-1)$
and $x\mapsto t(x)\in\mathcal{R}_{\alpha/(\alpha+r-1)}$.

If $f\in\mathcal{R}_r$ with $r>0$ and $q_+=\mathrm{ess}\sup Q<\infty$, then $k^\ast(z)\sim z q_+$ and
$$h_{ind}(x)\sim f\left(\frac{x}{q_+}\right).$$

\item[d)] 
$$h_{co}(x)=\inf_{t\geq 1}\left\{\max\left\{f(t),t\,k\left(\frac{x}{t}\right)\right\}\right\}$$
If $f\in\mathcal{R}_r$ and $k\in\mathcal{R}_\alpha$ with $r,\alpha>1$, then
$$h_{co}(x)\sim \frac{\alpha-1}{\alpha+r-1}\left(\frac{r}{\alpha-1}\right)^{r/(\alpha+r-1)} h_{ind}(x)$$
and $x\mapsto t(x)\in\mathcal{R}_{\alpha/(\alpha+r-1)}$.

\item[e)] If $f\in\mathcal{R}_r$ and $k\in\mathcal{R}_\alpha$ with $r,\alpha>1$ and $q_-=\mathrm{ess}\inf Q>0$, then
\begin{align}\label{hcount}
h_{counter}(x)\sim\min\left\{ f\left(\frac{x}{q_-}\right),\frac{k((1-m_-)x)}{1-m_-}\right\}\in\mathcal{R}_{\min\{r,\alpha\}} ,
\end{align}
where $m_-=\mathrm{ess}\inf M$. 
%
\end{itemize}
\end{thm}
\begin{remark}\label{afterH}
Function $t$ {satisfying \eqref{XT} is not unique, it} is not necessarily monotone nor may have a limit. An easy example may be constructed using $e)$, where
$t(x)\in\{t_1(x),t_2(x)\}$ and $t_1(x)\sim x/q_-$ and $t_2(x)\sim (1-m_-)^{-1}$.

Another important example can be constructed as follows. Let $\gamma>1$. Assume that $(M,Q)$ has an atom $\P(M=0,Q=1)=1-e^{-1}$ and an absolutely continuous part on $(0,1)\times(1,\infty)$ given by 
$$\P(M>x,Q>y)=\exp\left(-\frac{y^\gamma}{(1-x)^{\gamma-1}}\right)\qquad (x,y)\in[0,1)\times[1,\infty)$$
so that $\P(M>0,Q>1)=e^{-1}$. For $x>1$ we have
$$f(x)=-x\log\P(M>1-1/x)=x^\gamma\qquad\mbox{ and }\qquad k(x)=-x\log\P(Q>x)=x^\gamma.$$
If $M$ and $Q$ were independent, then we would have $h_{ind}\in \mathcal{R}_{\gamma^2/(2\gamma-1)}$. 
However, in our case they are not independent and it is easy to see that for any $x, t\geq 1$, 
$$-t\log\P\left(\frac{1}{1-M}>t,Q>\frac{x}{t}\right)=\max\{x,t\}^\gamma$$
so that $h(x)=x^\gamma$ for $x>1$ and 
$$h(x)=-t\log\P\left(\frac{1}{1-M}>t,Q>\frac{x}{t}\right)$$
for any $t=t(x)\in[1,x]$.
\end{remark}

\begin{remark}\label{R12}
If 
$$R=\sum_{k=1}^\infty M_1\cdot \ldots\cdot M_{k-1} Q_k,$$
then
\begin{align*}
R\geq R^{(1)}:=\sum_{k=1}^\infty m_-^{k-1} Q_k
\intertext{and (assume that $q_->0${)}}
R\geq R^{(2)}:=\sum_{k=1}^\infty M_1\cdot \ldots\cdot M_{k-1} q_-.
\end{align*}
Let $f$ and $k$ be defined as in Theorem \ref{propH} and assume that $f\in\mathcal{R}_r$ and $k\in\mathcal{R}_\alpha$ with $r,\alpha>1$. 
We have
$$\frac{\log\E e^{z R^{(1)}}}{\log\E e^{z Q}}=\sum_{k=1}^\infty \frac{\log\E e^{z m_-^{k-1} Q}}{\log\E e^{z Q}}.$$
Using regular variation of $\log\E e^{z Q}\sim k^\ast(z)$ and Potter bounds (\cite[Theorem~1.5.6]{BGT89}), we may pass with the limit under the sum to obtain
$$\lim_{z\to\infty}\frac{\log\E e^{z R^{(1)}}}{k^\ast(z)}=\sum_{k=1}^\infty m_-^{(k-1)\beta}=\frac{1}{1-m_-^\beta}.$$
Thus, by Kasahara's Theorem
\begin{align*}
\log\P(R>x)\geq\log\P(R^{(1)}>x)\sim -\sup_{z>0}\{zx-\frac{1}{1-m_-^\beta}k^\ast(z)\}=-\frac{k((1-m_-^\beta)x)}{1-m_-^\beta}.
\end{align*}
On the other hand, by \cite{BK17} we have
\begin{align*}
\log\P(R>x)\geq\log\P(R^{(2)}>x)\sim -\left(\frac{r}{r-1}\right)^{r-1}f\left(\frac{x}{q_-}\right).
\end{align*}
which gives by Theorem \ref{propH} e)
$$\liminf_{x\to\infty}\frac{\log\P(R>x)}{h_{counter}(x)}\geq -C$$
for some $C>0$.
In the next {s}ection we will give more accurate lower bound.
\end{remark}

\section{Lower bound}\label{secLower}
By {Theorem} \ref{propH} a) we know that there exists a function $t$ such that 
\begin{align}\label{htht}
h(x)=-t(x)\log\P\left(\frac{1}{1-M}>t(x),Q>\frac{x}{t(x)}\right)+o(1),
\end{align}
however function $t$ is not unique. Eye opener example was introduced in Remark \ref{afterH}, where we had 
$$h(x)=-t\log\P\left(\frac{1}{1-M}>t,Q>\frac{x}{t}\right)=x^\gamma \qquad \mbox{for all }t\in[1,x].$$

Below we present a lower bound for logarithmic asymptotics of the tail of $R$. Rate of convergence is described by the {regularly varying} function $h$, while the constant depends on the index of $h$ and the limit of a function $t$. If there is no uniqueness of function $t$, then the following result holds true for any such function provided that it converges to a limit {at infinity}.
\begin{thm}\label{lower}
Assume \eqref{MQpos}.
Assume that {function $h$ defined in \eqref{defh} belongs to $\mathcal{R}_\gamma$} with $\gamma\in[1,\infty]$.
If $\gamma=\infty$, assume additionally that $h\in\Gamma\subset \mathcal{R}_{\infty}$.

Finally, assume that $h$ is such that \eqref{htht} holds for {a function $t$} with
$\lim_{x\to\infty}t(x)=t_\infty\in(1,\infty]$.

Then,
\begin{align}\label{super}
\liminf_{x\to\infty} \frac{\log\P(R> x)}{h(x)}\geq 
-c_{t_\infty,\gamma},
\end{align}
where $c_{t,\gamma}$ is a finite positive constant given below; if $t\in(1,\infty)$ and $\gamma\in(1,\infty)$, then
\begin{align}
c_{t,1}&=c_{\infty,1}=1,\label{gamma1} \\
c_{t,\gamma}&=\left[t\left\{1-\left(1-\frac{1}{t} \right)^{\gamma/(\gamma-1)}\right\} \right]^{\gamma-1}, \label{BIGeq2} \\
\intertext{{otherwise,}}
c_{\infty,\gamma}&= \left(\frac{\gamma}{\gamma-1}\right)^{\gamma-1} \label{tinfty1},\\
c_{\infty,\infty}&=e \label{tinfty2},\\
c_{t,\infty}&=\left(1+\frac{1}{t}\right)^{1+t} \label{allinfty}.
\end{align}
\end{thm}
	
	\begin{example}
		Let us consider a perpetuity $R$ generated by $(M,Q)$ such that 
$\P(M=m)=1$ with $m\in(0,1)$ and $x\mapsto -\log\P(Q>x)=:k(x)\in\mathcal{R}_\alpha$ with $\alpha>1$. Then
	we have	$t(x)=t_\infty=\frac{1}{1-m}$ and 
		$$h(x)=-t_\infty\log\P(Q> x/t_\infty)\sim t_\infty^{1-\alpha}k(x).$$	
On the other hand, (by calculations from Remark \ref{R12})
$$\log\P(R>x)\sim -(1-m^\beta)^{\alpha-1}k(x) \sim -(1-m^\beta)^{\alpha-1} t_\infty^{\alpha-1} h(x)$$
with $\beta=\alpha/(\alpha-1)$. Finally, we see that
$$(1-m^\beta)^{\alpha-1} t_\infty^{\alpha-1}=c_{t_\infty,\gamma}$$
where $\gamma=\alpha$. This means that the constant obtained in \eqref{BIGeq2} is optimal.
	\end{example}
	
\begin{proof}[Proof of Theorem~\ref{lower}]
Without loss of generality we  may assume that $h$ is differentiable and, if $\gamma>1$, ultimately convex. {For $\gamma\in[1,\infty)$ use Smooth Variation Theorem, for $\gamma=\infty$ use arguments from page 5 in \cite{BK17}.}
\begin{description}
\item[Case $t_\infty<\infty$ and $\gamma=1$]
Observe that on the set 
$$\bigcap_{k=1}^n \{ M_k>1-\delta, Q_k>q\}$$
we have
$$R\geq \sum_{k=1}^n M_1\cdot\ldots\cdot M_{k-1}Q_k> q \frac{1-(1-\delta)^n}{\delta},$$
which means that for any $\delta\in(0,1)$, $q>0$ and $n\in\mathbb{N}$ we have
\begin{align}\label{firstcase}
\log\P\left(R>q \frac{1-(1-\delta)^n}{\delta}\right) \geq \log\P\left( \bigcap_{k=1}^n \{ M_k>1-\delta, Q_k>q\}\right)=n\log\P(M>1-\delta,Q>q).
\end{align}
For given $x>0$, set 
$$\delta=\delta(x)=\frac{1}{t(x)}\qquad\mbox{and}\qquad q=q(x)=\frac{x}{t(x)}\qquad\mbox{and}\qquad n=1$$
 so that
 $$\log\P(M>1-\delta(x),Q>q(x))\sim -\frac{h(x)}{t_\infty}$$
 and
 $$q \frac{1-(1-\delta)^n}{\delta}=\frac{x}{t(x)}.$$

Then, \eqref{firstcase} gives
$$\frac{\log\P\left(R>\frac{x}{t(x)}\right)}{h\left(\frac{x}{t(x)}\right)}\geq 
\frac{\log\P(M>1-\delta(x),Q>q(x))}{h(x)}\frac{h(x)}{h\left(\frac{x}{t(x)}\right)}\sim -\frac{1}{t_\infty} \frac{1}{1/t_\infty}=-1.$$
We will show that this implies $\liminf_{x\to\infty}{\log\P\left(R>x\right)}/{h\left(x\right)}\geq -1$.
Let {$x_0$} be such that $t(x)/t_\infty\in(1-\eps,1+\eps)$ for {$\eps\in(0,1)$ and} all {$x>x_0$}. Then $x/(t_\infty(1+\eps))\leq x/t(x)\leq x/(t_\infty(1-\eps))$ and
\begin{align}\label{lowerTg}
\frac{\log\P\left(R>\frac{x}{t(x)}\right)}{h\left(\frac{x}{t(x)}\right)}\leq \frac{\log\P\left(R>\frac{x}{t_\infty(1+\eps)}\right)}{h\left(\frac{x}{t_\infty(1-\eps)}\right)}
\end{align}
for {$x>x_0$}, thus
\begin{align*}
\liminf_{x\to\infty} \frac{\log\P\left(R>x\right)}{h\left(x\right)}=\liminf_{x\to\infty} \frac{\log\P\left(R>\frac{x}{t_{\infty}(1+\eps)}\right)}{h\left(\frac{x}{t_\infty(1+\eps)}\right)} 
\geq \liminf_{x\to\infty} \frac{\log\P\left(R>\frac{x}{t(x)}\right)}{h\left(\frac{x}{t(x)}\right)} \frac{h\left(\frac{x}{t_\infty(1-\eps)}\right)}{h\left(\frac{x}{t_\infty(1+\eps)}\right)} \geq
-1\cdot\tfrac{1+\eps}{1-\eps}
\end{align*}
by \eqref{lowerTg} and regular variation of $h$. Passing with $\eps\to0$, we obtain the first part of \eqref{gamma1}. 


\item[Case $t_\infty=\infty$ and $\gamma=1$]
We proceed similarly as in the previous case. For arbitrary $\alpha>0$ set
$$\delta=\frac{1}{t(x)}\qquad\mbox{and}\qquad q=\frac{x}{t(x)}\qquad\mbox{and}\qquad n=\left\lfloor \alpha t(x)\right\rfloor$$
in \eqref{firstcase} to obtain for any $x>0$,
$$\frac{\log\P\left(R>x\left(1-\left(1-\tfrac{1}{t(x)}\right)^{\left\lfloor \alpha t(x)\right\rfloor}\right)\right)}{h\left(x\left(1-\left(1-\tfrac{1}{t(x)}\right)^{\left\lfloor \alpha t(x)\right\rfloor}\right)\right)} \geq
\frac{n\log\P(M>1-\delta,Q>q)}{h(x)}\frac{h(x)}{h\left(x\left(1-\left(1-\tfrac{1}{t(x)}\right)^{\left\lfloor \alpha t(x)\right\rfloor}\right)\right)}.$$
Since $t(x)\to\infty$ as $x\to\infty$, by regular variation of $h$, we see that the right hand side converges to 
$$-\frac{\alpha}{1-e^{-\alpha}}.$$
Using similar approach as in the case $t_\infty<\infty$, we show that
$$\liminf_{x\to\infty}\frac{\log\P(R>x)}{h(x)}\geq -\frac{\alpha}{1-e^{-\alpha}}.$$
Passing to the limit with $\alpha\to0$, we obtain the second part of \eqref{gamma1}. 


\item[Case $t_\infty<\infty$ and $\gamma\in(1,\infty)$]

For given $n\in\mathbb{N}$, consider sequences $(\delta_k)_{k=1}^n$ and $(q_k)_{k=1}^n$ satisfying
\begin{align}\label{subject}
x{\leq}\sum_{k=1}^n(1-\delta_1)\cdot\ldots\cdot(1-\delta_{k-1}) q_k.
\end{align}
Then, we have
$$\log\P(R>x)\geq \sum_{k=1}^n\log\P(M>1-\delta_k,Q>q_k)$$
and so
$$
\frac{\log\P(R>x)}{h(x)} \geq  \sum_{k=1}^n \frac{\log\P(M>1-\delta_k,Q>q_k)}{h(x)}.
$$
Set {for $k=1,\ldots,n$}
$$y_k={u_k}x\quad\mbox{and}\quad\delta_k=\frac{1}{t(y_k)}\quad\mbox{and}\quad q_k= \frac{y_k}{t(y_k)}{,}$$
{where $u_1,\ldots,u_n$ are some positive constants }
{such that (compare with \eqref{subject})}
\begin{align}\label{subjectX}
1{\leq}\sum_{k=1}^n \pi_{k}(x){u_k},
\end{align}
where
$$\pi_{k}(x)=(1-\delta_1)\cdot\ldots\cdot(1-\delta_{k-1})\frac{1}{t(y_k)}\to \left(1-\frac{1}{t_\infty}\right)^{k-1}\frac{1}{t_\infty}\qquad\mbox{as }x\to\infty,$$
since $y_i\to\infty$ for $i=1,\ldots,{n}$.
{Passing as $x\to\infty$ in the right hand side of \eqref{subjectX} we obtain
\begin{align}\label{RHSu}\frac{1}{t_\infty} \sum_{k=1}^n \left(1-\frac{1}{t_\infty}\right)^{k-1}u_k.\end{align}
We will choose $(u_k)_k$ in such a way that the above expression is strictly greater then $1$ and this will ensure that \eqref{subjectX} holds for large $x$.
Let us consider 
\begin{align}
\label{defuk}
u_k=t_{\infty}(1-t_\infty^{-1})^{1-k} A B^{k-1},\qquad k=1,\ldots,n
\end{align}
for positive $A$ and $B\in(0,1)$. Inserting it into \eqref{RHSu} we get
$$A\frac{1-B^n}{1-B}.$$
If additionally $A>1-B$ there exists $N$ such that for all $n\geq N$ the above expression is strictly larger than $1$. Thus, \eqref{subjectX} is established for such a choice.
}
Moreover, by the definition of $h$ and function $t$, we have for any $\eps>0$ {and $x>0$},
$$h(x)\leq-t(x)\log\P(M>1-1/t(x),Q>x/t(x))\leq (1+\eps)h(x)$$
and so
\begin{align}\label{defa}
\frac{\log\P(R>x)}{h(x)} \geq
-(1+\eps)\sum_{k=1}^n \frac{h\left( {u_k}x\right)}{t(y_k)h(x)}.
\end{align}
{Taking $\liminf_{x}$ of both sides of \eqref{defa}, we obtain for any $n\geq N$
$$\liminf_{x\to\infty}\frac{\log\P(R>x)}{h(x)}\geq-\frac{1+\eps}{t_\infty}\sum_{k=1}^n u_k^\gamma $$
and passing with $n\to\infty$ along with the substitution of \eqref{defuk} we obtain
$$\liminf_{x\to\infty}\frac{\log\P(R>x)}{h(x)}\geq-(1+\eps)t_{\infty}^{\gamma-1} \frac{A^\gamma}{1-\left(B \frac{t_{\infty}}{t_{\infty}-1}\right)^\gamma}.$$
The above inequality holds for any $A>1-B\in(0,1)$. Let us set $A=1-B+\eps$. Then the expression on the right hand side above attains its supremum for 
$$B_{\eps}=(1+\eps)^{1/(1-\gamma)}\left(1-\frac{1}{t_{\infty}}\right)^{\gamma/(\gamma-1)}$$
and for such $B$ this supremum equals 
$$-t_{\infty}^{\gamma-1}(1-B_{\eps}+\eps)^{\gamma-1}.$$
Letting  $\eps\to0$ we obtain 
\eqref{BIGeq2}. 
}


\item[Case $t_\infty=\infty$ and \mbox{$\gamma\in(1,\infty]$}]
Let {$x_0=0$ and $R_0=Q_0$} and define a random sequence $(R_n)_{n\geq1}$ and a sequence of scalars $(x_n)_{n\geq1}$ through
$$R_n=M_n R_{n-1}+Q_n\qquad \mbox{and}\qquad x_n=(1-\delta_n)x_{n-1}+q_n,\qquad  n\geq1,$$
where $(M_n,Q_n)_{n\geq{0}}$ is an i.i.d. sequence of the generic element $(M,Q)$ and $(\delta_n)_{n\geq1}$ and $(q_n)_{n\geq1}$ are scalar sequences yet to be determined.

Since $M$ and $Q$ are assumed to be a.s. non-negative and 
{$$R=\sum_{k=1}^\infty Q_k\prod_{j=1}^{k-1}M_j\geq\sum_{k=1}^{n+1}Q_k\prod_{j=1}^{k-1}M_j\stackrel{d}{=}R_n,$$}
we have 
$$\P(R> x)\geq \P(R_n> x).$$
Moreover, since $(M_n,Q_n)$ and $R_{n-1}$ are independent, we have
\begin{align}\label{kuku}\begin{split}
\P(R_n> x_n)&\geq \P(M_n R_{n-1}+Q_n> (1-\delta_n)x_{n-1}+q_n,M_n> 1-\delta_n, Q_n> q_n) \\
&\geq \P(M_n> 1-\delta_n, Q_n> q_n)\P(R_{n-1}> x_{n-1})\geq \prod_{k=1}^n \P(M> 1-\delta_k, Q> q_k){\P(Q>0)}
\end{split}\end{align}
{and $\P(Q>0)>0$.}
If $(x_n)_n$ is strictly increasing and if $x_{n-1}< x\leq x_{n}$, then 
$$\frac{\log\P(R>x)}{h(x)}\geq\frac{\log\P(R>x_n)}{h(x_{n-1})}$$ 
and therefore, if additionally $(x_n)_n$ is divergent and $h(x_n)/h(x_{n-1})$ has a limit as $n\to\infty$, we have
\begin{align*}
\liminf_{x\to\infty} \frac{\log\P(R>x)}{h(x)}
&\geq
\liminf_{n\to\infty} \frac{\log\P(R> x_n)}{h(x_{n-1})}  \\
&\stackrel{\eqref{kuku}}{\geq }
\liminf_{n\to\infty} \frac{h(x_n)}{h(x_{n-1})}\frac{\sum_{k=1}^n \log\P(M> 1-\delta_k, Q> q_k)}{h(x_n)}  \\
& \geq
\lim_{n\to\infty}\frac{h(x_n)}{h(x_{n-1})} \liminf_{n\to\infty} \frac{\sum_{k=1}^n \log\P(M> 1-\delta_k, Q> q_k)}{h(x_n)}  \\
& \geq
\lim_{n\to\infty}\frac{h(x_n)}{h(x_{n-1})} \liminf_{n\to\infty} \frac{\log\P(M> 1-\delta_n, Q> q_n)}{h(x_n)-h(x_{n-1})}, 
\end{align*}
where the last inequality follows by the Stoltz--Ces\`{a}ro theorem {(recall that $h(x)\to\infty$ as $x\to\infty$)}.

We will choose now sequences $(\delta_n)_{n\geq1}$ and $(q_n)_{n\geq1}$ in such a way that $x_n\to\infty$ and above limit is finite and negative.

Let us set
	$$I_n:= \frac{\log\P(M> 1-\delta_n,Q> q_n)}{h(x_n)-h(x_{n-1})}$$
	and let 
	$$\delta_n=\frac{1}{t(y_n)}\qquad\mbox{ and }\qquad q_n=\frac{y_n}{t(y_n)},$$ 
	where 
	$$y_n=x_{n-1}+c\frac{h(x_{n-1})}{h^\prime(x_{n-1})}$$
	for some positive constant $c$.
Inserting {the} above into the definition of $(x_n)_{n\geq1}$ we obtain
\begin{align}\label{difx}
x_n-x_{n-1}=q_n-\delta_n x_{n-1}=\frac{c}{t(y_n)}\frac{h(x_{n-1})}{h^\prime(x_{n-1})}.
\end{align}

		Since the right hand side of \eqref{difx} is positive, $x_n$ is strictly increasing. This means that $x_n$ has a limit, possibly infinite. Assume that $p:=\lim_n x_n<\infty$. Then $y_n\to p+c\,h(p)/h^\prime(p)<\infty$ and, by \eqref{difx}, we see that 
		$$0=\lim_{n\to\infty}\frac{c}{t(y_n)}\frac{h(x_{n-1})}{h^\prime(x_{n-1})}=\frac{c\,h(p)}{h^\prime(p)}\lim_{n\to\infty}\frac{1}{t(y_n)}.$$
		But this is impossible, because for any finite $x>0$, $t(x)$ is finite (Theorem \ref{propH} a)). Thus $x_n\to\infty$.

Further,
	$$I_n=-C_n\frac{h(y_n)}{t(y_n)(h(x_n)-h(x_{n-1}))},$$
	where 
	$$C_n:=\frac{-t(y_n)\log\P\left(M> 1-\frac{1}{t(y_n)},Q> \frac{y_n}{t(y_n)}\right) }{h(y_n)}\to1\qquad\mbox{ as }n\to\infty.$$
Using convexity of $h$, we obtain
	$$I_n\geq - C_n\frac{h(y_n)}{t(y_n)(x_n-x_{n-1}) h^\prime(x_{n-1})}=- C_n\frac{h\left(x_{n-1}+c\frac{h(x_{n-1})}{h^\prime(x_{n-1})}\right)}{c h(x_{n-1})}.$$
	
	Letting $n\to\infty$, we have (see \eqref{Rg} and \eqref{Gamma})
	$$\liminf_{n\to\infty}I_n\geq \begin{cases}
	-\frac{1}{c}\left(\frac{c+\gamma}{\gamma}\right)^\gamma & \mbox{if }\gamma\in[1,\infty) \\
	-\frac{e^c}{c}& \mbox{if }\gamma=\infty.	
	\end{cases}
	$$
	If $\gamma\in(1,\infty)$, then the supremum of the right hand side above is attained at $c=\gamma/(\gamma-1)$ and this supremum equals $-(\gamma/(\gamma-1))^{\gamma-1}$. 
  For $\gamma=\infty$, the supremum is attained at $c=1$ and then equals $-e$.
It is left to show that
$\lim_{n\to\infty}\frac{h(x_n)}{h(x_{n-1})}=1$.
We have
$$
\frac{h(x_n)}{h(x_{n-1})}=\frac{h\left(x_{n-1}+\frac{c}{t(y_n)}\frac{h(x_{n-1})}{h^\prime(x_{n-1})}\right)}{h(x_{n-1})}\to1,
$$
since $\lim_{n\to\infty} t(y_n)=\infty$ (the convergence in \eqref{Gamma} is uniform; see \cite[Proposition~3.10.2]{BGT89}).


\item[Case $t_\infty<\infty$ and $\gamma=\infty$]
Proceeding in the same way as in the previous case, we obtain
$$\liminf_{x\to\infty} \frac{\log\P(R>x)}{h(x)}\geq \lim_{n\to\infty}\frac{h(x_n)}{h(x_{n-1})} \liminf_{n\to\infty} I_n,$$
where $x_n\to\infty$ and
\begin{align*}
	I_n= -\frac{C_n}{t(y_n)}\frac{h(x_{n-1}+c\frac{h(x_{n-1})}{h^\prime(x_{n-1})})}{h\left(x_{n-1}+\frac{c}{t(y_n)}\frac{h(x_{n-1})}{h^\prime(x_{n-1})}\right)-h(x_{n-1})},
\end{align*}
where $C_n\to1$ as $n\to\infty$. Thus, using \eqref{Gamma}, we obtain
$$
\lim_{n\to\infty} I_n = -\frac{e^c}{t_\infty\left(e^{c/t_\infty}-1\right)}
$$
and
$$
\lim_{n\to\infty}\frac{h(x_n)}{h(x_{n-1})}=	e^{c/t_\infty}.
$$
Thus, 
$$\liminf_{x\to\infty} \frac{\log\P(R>x)}{h(x)}\geq-\inf_{c>0}\left\{ e^{c/t_\infty}\frac{e^c}{t_\infty\left(e^{c/t_\infty}-1\right)}\right\}=-\left(1+\frac{1}{t_\infty}\right)^{t_\infty+1}.$$
\end{description}
\end{proof}

\begin{remark}
In example introduced in Remark \ref{afterH}, we have $h\in\mathcal{R}_\gamma$ with $\gamma\in(1,\infty)$ and
$$h(x)\sim-t\log\P\left(\frac{1}{1-M}>t,Q>\frac{x}{t}\right)$$
for any $t\in(1,\infty)$, so that Theorem \ref{lower} gives us for any $t>1$
$$\liminf_{x\to\infty}\frac{\log\P(R>x)}{h(x)}\geq -\left[t\left\{1-\left(1-\frac{1}{t} \right)^{\gamma/(\gamma-1)}\right\} \right]^{\gamma-1}.$$
We have 
$$\inf_{t>1}\left[t\left\{1-\left(1-\frac{1}{t} \right)^{\gamma/(\gamma-1)}\right\} \right]^{\gamma-1}=1$$ 
so that
$$
\liminf_{x\to\infty}\frac{\log\P(R>x)}{h(x)}\geq -1.
$$
\end{remark}

Below, we give an example of two perpetuities of which generators have the same marginals, while perpetuities have logarithmic tails of different asymptotic order.
\begin{example}\label{EXQQQ}
Let $X=(X(t))_{t\geq 0}$ be a drift-free non-killed subordinator
with Laplace exponent $\Phi(s)=-\log\E e^{-sX(1)}$, $s\geq 0$ and
$T$ an exponentially distributed random variable of parameter $1$ which is independent of $X$. The random variable $R:=\int_0^\infty
e^{-X(t)}{\dd}t$ is a perpetuity generated by 
$$(M,Q):=\left(e^{-X(T)}, \int_0^Te^{-X(t)}{\dd}t\right).$$ 
One is able to give semi-explicit formula for joint moments of $M$ and $Q$ (see formula (2.6) in \cite{IM08}).

Assume now that $\Phi\in\mathcal{R}_\alpha$ with $\alpha\in(0,1)$.
Then, it is proven in \cite{Riv03} that
\begin{align}\label{EX}
-\log\P\{R>x\}\sim(1-\alpha)\Psi(x),\quad t\to\infty
\end{align}
with $\Psi(x):=\inf\{s>0: \frac{s}{\Phi(s)}>x\}$.

If one takes the L\'{e}vy measure of $X$ of the form
$$\nu(\dd t)=\frac{e^{-t/\alpha}}{(1-e^{-t/\alpha})^{\alpha+1}}I_{(0,\infty)}(t)\dd t,$$ 
then
$$\Phi(s)=\int_{[0,\infty)} (1-e^{-st})\nu(\dd t)=\frac{\Gamma(1-\alpha)\Gamma(1+\alpha s)}{\Gamma(1+\alpha(s-1))}-1\sim \alpha^\alpha \Gamma(1-\alpha)s^{\alpha}$$
and one can find marginal distributions of $(M,Q)$ (see Example 2.1.2 in \cite{IksBook16}).
In this special case, $Q$ has Mittag-Leffler distribution with parameter $\alpha$ and $M^{1/\alpha}$ has beta distribution with parameters $1-\alpha$ and $\alpha$. 
This implies that 
$$f(x):=-x\log\P\left(M>1-\frac{1}{x}\right)\sim \alpha x\log x,$$
and (see e.g. \cite[Theorem 8.1.12]{BGT89})
$$k(x):=-\log\P(Q>x)\sim (1-\alpha)\alpha^{\alpha/(1-\alpha)}x^{1/(1-\alpha)},$$
that is, $k\in\mathcal{R}_{1/(1-\alpha)}$ and $f\in\mathcal{R}_1$. 

Let us consider now a perpetuity $R_{ind}$ generated by independent $M$ and $Q$ with the same distributions. One can show that the corresponding function $h_{ind}=(f^\ast\circ k^\ast)^\ast$ belongs to $\mathcal{R}_1$. Thus, by Theorem \ref{lower} we have 
$$-\log\P(R_{ind}>x)\lesssim C h_{ind}(x)$$
for some $C>0$.
On the other hand, \eqref{EX} implies that $x\mapsto-\log\P\{R>x\}$ is regularly varying at $\infty$ {with} index $(1-\alpha)^{-1}>1$ {($\Psi=\rho^\leftarrow$, where $\rho(s)=s/\Phi(s)\in\mathcal{R}_{1-\alpha}$)}.
Therefore, we obtain
$$\lim_{x\to\infty}\frac{\log\P(R_{ind}>x)}{\log\P(R>x)}=0.$$
\end{example}

\section{Upper bound}\label{secUpper}
In the {present} {s}ection we give asymptotic upper bounds for $\log\P(R>x)$ when $(M,Q)$ is negatively quadrant dependent (Theorem \ref{NQD}) and when $(M,Q)$ is dependent in an arbitrary way (Theorem \ref{THco}, which is the most important result of this {s}ection. 

Let us assume that
\begin{align}\label{kfdef}
k(x)=-\log\P(Q>x)\in\mathcal{R}_\alpha, \qquad f(x)=-x\log\P(M\geq 1-1/x)\in\mathcal{R}_{r}\qquad\mbox{with }\alpha,r>1.
\end{align}
Let $r^\ast$ and $\beta$ denote the conjugate numbers to $r$ and $\alpha$, respectively and denote
$$\gamma=\beta r^\ast/(\beta r^\ast-1).$$
Let $R_{co}$ and $R_{ind}$ denote perpetuities generated by comonotonic and independent $(M,Q)$, respectively, and let $h_{co}$ and $h_{ind}$ denote $h$ functions corresponding to these two cases. Recall that in Theorem \ref{propH} we have shown that 
\begin{align}\label{seqh}
(f^\ast\circ k^\ast)^\ast(x)\sim h_{ind}(x)\sim \frac{\alpha+r-1}{\alpha-1} \left(\frac{\alpha-1}{r}\right)^{r/(\alpha+r-1)}h_{co}(x).
\end{align}
\begin{thm}\label{THco}
	Assume \eqref{MQpos} and \eqref{kfdef}.
Then,
\begin{align}\label{co1}
\limsup_{x\to\infty}\frac{\log\P(R>x)}{h_{co}(x)} \leq -\left(\frac{\gamma}{\gamma-1}\right)^{\gamma-1}
\intertext{and}
\label{co}
\lim_{x\to\infty}\frac{\log\P(R_{co}>x)}{h_{co}(x)}= -\left(\frac{\gamma}{\gamma-1}\right)^{\gamma-1}.
\end{align}
\end{thm}

If we additionally assume that $(M,Q)$ is negatively or positively quadrant dependent, then we can prove slightly stronger results.
\begin{thm}\label{NQD}
	Assume \eqref{MQpos} and \eqref{kfdef}.
\begin{itemize}
\item[\rm{(i)}]
 If $(M,Q)$ is negatively quadrant dependent, then
\begin{align}\label{NQD1}
\limsup_{x\to\infty}\frac{\log\P(R>x)}{h_{ind}(x)} \leq -\left(\frac{\gamma}{\gamma-1}\right)^{\gamma-1}.
\end{align}
\item[\rm{{(ii)}}]If $(M,Q)$ is positively quadrant dependent, then
\begin{multline}\label{PQDl}
-\left(\frac{\gamma}{\gamma-1}\right)^{\gamma-1}
\leq
\liminf_{x\to\infty}\frac{\log\P(R>x)}{h_{ind}(x)} \\
\leq 
\limsup_{x\to\infty}\frac{\log\P(R>x)}{h_{ind}(x)}
\leq 
-\frac{\alpha-1}{\alpha+r-1}\left(\frac{r}{\alpha-1}\right)^{r/(\alpha+r-1)} \left(\frac{\gamma}{\gamma-1}\right)^{\gamma-1}.
\end{multline}
\end{itemize}
\end{thm}

\begin{proof}[Proof of Theorem~\ref{NQD}]
Part \rm{(i)}. Since $r\mapsto \exp(zr)$ is convex, by Lemma~\ref{LB} we see that 
$$\E e^{z R}\leq \E e^{z R_{ind}},\qquad z\geq 0,$$
Let $\psi_{ind}(z):=\log\E \exp(zR_{ind})$.
By exponential Markov inequality 
$$\P(R>x)\leq \frac{\E e^{z R}}{e^{z x}}\leq \frac{\E e^{z R_{ind}}}{e^{z x}}.$$
After taking $\log$ and $\inf_{z>0}$ of both sides we arrive at
\begin{align*}
\log \P(R>x)\leq -\psi_{ind}^\ast(x).
\end{align*}
By Kasahara's Tauberian Theorem we conclude that $-\psi_{ind}^\ast(x)\sim -\log\P(R_{ind}>x)$ and thus {Theorem~\ref{thmIND1}} gives us the assertion.

Part \rm{(ii)}. The upper bound follows by Theorem~\ref{THco}.
The lower bound in \eqref{PQDl} is immediate if one looks into the proof of Theorem~\ref{lower}. 
By positive quadrant dependence, we have
$$\log\P(M>\delta_k,Q>q_k)\geq \log\left[\P(M>\delta_k)\P(Q>q_k)\right].$$
Thus, using above inequality and repeating all the steps of the proof of Theorem~\ref{lower} {(case $t_\infty=\infty$)} with $h(x)=(f^\ast\circ k^\ast)^\ast(x)\in\mathcal{R}_\gamma$, we arrive at the lower bound in \eqref{PQDl}.
\end{proof}

{For proving \eqref{co1} w}e will need the following Lemma, whose proof is postponed to {Section \ref{secProof}}.
\begin{lemma}\label{lem2}
	Assume that {there exists a function $\phi$ such that}
	$$I_{\phi}(z):=\E \exp(zQ+ \phi(z M)-\phi(z))\leq 1$$
	for large values of $z$. Then there exists a constant $C>0$ such that
	$$\E \exp(z R)\leq \exp(\phi(z)+C)$$
	for $z$ large enough.
\end{lemma}

\begin{proof}[Proof of Theorem~\ref{THco}]
Observe that \eqref{co} follows by \eqref{co1}. 
Indeed, by Theorem \ref{lower} {(see \eqref{tinfty1})}, we have
$$\liminf_{x\to\infty}\frac{\log\P(R_{co}>x)}{h_{co}(x)}\geq - \left(\frac{\gamma}{\gamma-1}\right)^{\gamma-1}.$$

Since $r\mapsto \exp(zr)$ is convex {and monotonic}, by Lemma \ref{LB} we see that 
$$\E \exp(z R)\leq \E \exp(z R_{co}),\qquad z\geq 0,$$
where 
$$R_{co}\stackrel{d}{=}\overline{M}R_{co}+\overline{Q},\qquad(\overline{M},\overline{Q})\mbox{ and }R_{co}\mbox{ are independent}$$
and $(\overline{M},\overline{Q})\stackrel{d}{=}(F_M^{-1}(U),F_Q^{-1}(U))$, $U{\stackrel{d}{\sim}}\mathrm{U}([0,1])$, is a comonotonic vector with given marginals.
Let $\overline{\psi}(z):=\log\E \exp(zR_{co})$.
By exponential Markov inequality 
$$\P(R>x)\leq \frac{\E e^{z R}}{e^{z x}}\leq \frac{\E e^{z R_{co}}}{e^{z x}}.$$
After taking $\log$ and $\inf_{z>0}$ of both sides we arrive at
\begin{align}\label{firstineq}\log \P(R>x)\leq -\overline{\psi}^\ast(x).\end{align}

By the Smooth Variation Theorem, there exist $\lf,\uf\in\mathcal{SR}_r$ and $\lk,\uk\in\mathcal{SR}_\alpha$ with 
$$\lf(x)\sim\uf(x)\qquad\mbox{and}\qquad\lk(x)\sim\uk(x)$$
and 
$$\lf\leq f\leq\uf\qquad\mbox{and}\qquad \lk\leq k\leq\uk$$
on a neighbourhood of infinity. 
Define
{$$\phi=\lf^\ast\circ \lk^\ast\qquad\mbox{ and }\qquad \phi_B(x)=B\phi(x/B)\mbox{ for }B>0$$} 
and
$$B_{co}=\frac{\alpha-1}{\alpha+r-1}\left(\frac{r}{\alpha-1}\right)^{r/(\alpha+r-1)} \left(\frac{\gamma}{\gamma-1}\right)^{\gamma-1}.$$
Assume for a while that for any $B\in(0,B_{co})$,
\begin{align}\label{defIB}
I_B(z):=\E e^{z \overline{Q}+\phi_B(z\overline{M})-\phi_B(z)}\to 0\qquad\mbox{ as }z\to\infty.
\end{align}
By Lemma \ref{lem2} this implies that for any $B<B_{co}$,
$$\overline{\psi}(z)\leq \phi_B(z)+C_B$$
for large $z$ and some constant $C_B$.
Since convex conjugation is order reversing, we have
$$\overline{\psi}^\ast(x)\geq (\phi_B+C_B)^\ast(x)=\phi_B^\ast(x)-C_B.$$
Moreover, 
$$\phi_B^\ast(x)= B\phi^\ast(x)=B (\lf^\ast\circ \lk^\ast)^\ast(x).$$
Above, together with \eqref{firstineq} imply that for any $B<B_{co}$ we have
$$
\limsup_{x\to\infty}\frac{\log\P(R>x)}{\phi^\ast(x)}\leq-B.
$$
Passing to the limit as $B\uparrow B_{co}$, we obtain that
$$
\limsup_{x\to\infty}\frac{\log\P(R>x)}{(f^\ast\circ k^\ast)^\ast(x)}\leq-\frac{\alpha-1}{\alpha+r-1}\left(\frac{r}{\alpha-1}\right)^{r/(\alpha+r-1)} \left(\frac{\gamma}{\gamma-1}\right)^{\gamma-1}
$$
and by \eqref{seqh} this is equivalent to \eqref{co1}.

It is left to show that \eqref{defIB} holds.

For any $\eps\in(0,1)$ we have
	\begin{align}\label{IB1}
	I_B(z)= \E e^{zQ+\phi_B(zM)-\phi_B(z)}I_{M\leq 1-\eps}+\E e^{z\overline{Q}+\phi_B(z\overline{M})-\phi_B(z)}I_{M> 1-\eps}=:K_1(z)+K_2(z).
	\end{align}
Since $\psi_B\in \mathcal{R}_{\beta r^\ast}$ and, by Kasahara's Tauberian Theorem, $z\mapsto\log\E\exp(zQ)\sim k^\ast(z)\in\mathcal{R}_\beta$, we have
$$K_1(z)\leq e^{\log\E\exp(zQ)+\phi_B(z(1-\eps))-\phi_B(z)}=o(1).$$

By definition of the generalized inverse, we have
\begin{align*}
U\leq F_M(F_M^{-1}(U))\qquad\mbox{ and }\qquad F_M^{-1}(U)\stackrel{d}{=}M. 
\end{align*}
Thus,
\begin{align*}
K_2&=\E e^{z F_Q^{-1}(U)+\phi_B(zF_M^{-1}(U))-\phi_B(z)}I_{F_M^{-1}(U)> 1-\eps}\\
&\leq \E e^{z F_Q^{-1}\left(F_M(M)\right)+\phi_B\left(zM\right)-\phi_B(z)}I_{M>1-\eps}.
\end{align*}

Let us denote $s(x):={F}_Q^{-1}\left(F_M(x)\right)$. 
By the definition of $f$ and $k$ {we have for $x\in(0,1)$,
\begin{align*}
(1-x)f\left(\frac{1}{1-x}\right)=-\log(1-F_M(x)) \qquad\mbox{and}\qquad
F_Q^{-1}(x)\leq \lk^{-1}(-\log(1-x)).
\end{align*}
}
{Hence,} it is easy to see that in a left neighbourhood of $1$, we have $s\leq \overline{s}$,
where
\begin{align}\label{defs}
\overline{s}(x):=\lk^{-1}\left((1-x)\uf\left(\frac{1}{1-x}\right)\right).
\end{align}
Since $\phi$ is ultimately convex, we have for $x\in(1-\eps,1]$
$$\phi_B(z x)-\phi_B(z)\leq -z \phi_B^\prime(z x)(1-x)\leq -z \phi_B^\prime(z (1-\eps))(1-x)=-z \phi_{B_\eps}^\prime(z)(1-x),$$
where $B_\eps:=B/(1-\eps)$. Thus,
\begin{align*}
K_2(z)&\leq 
	\int_{(1-\eps,1]} \exp\left( z \overline{s}(x)-z \phi_{B_\eps}^\prime(z)(1-x)\right) \dd F_M(x)\\
&= \int_{(1-\eps,1]} \exp\left( z \overline{s}(x)-z \phi_{B_\eps}^\prime(z)(1-x)-\eta (1-x)f\left(\frac{1}{1-x}\right)\right) \frac{1}{(1-F_M(x))^\eta}\dd F_M(x) 
\intertext{since $\log\P(M>x)=-(1-x)f(1/(1-x))$. Further, }
K_2(z) & \leq \exp
\left(
    \sup_{t\in[0,\eps)}
    \left\{
          z \overline{s}(1-t)-zt \phi_{B_\eps}^\prime(z)-\eta t\lf\left(\frac1t\right)
    \right\}
\right)
\int_{(1-\eps,1]}\frac{\dd F_M(x)}{(1-F_M(x))^\eta}
\end{align*}
and the integral is finite for any $\eta\in(0,1)$.

	Since all functions involved are smooth, {one can show that for $\eps$ small enough, the expression under $\sup$ as a function of $t\in(0,\eps)$ is concave (calculate the second derivative and use the fact that $x\mapsto \overline{s}(1-1/x)\in \mathcal{R}_{(r-1)/\alpha}$ and $x\mapsto \lf(x)/x\in\mathcal{R}_{r-1}$). Hence,} the supremum above is attained at $t_0=t_0(z)$ such that
	\begin{align}\label{deft0}
	z \overline{s}^\prime(1-t_0)+z \phi_{B_\eps}'(z)=\eta\tfrac{1}{t_0}\lf'\left(\tfrac{1}{t_0}\right)-\eta\lf\left(\tfrac{1}{t_0}\right).
	\end{align}
	Put $t_0=1/(\lf^\ast)^\prime(\lk^\ast(x))$.
	Then, by \eqref{eqstar},
	$$z\overline{s}^\prime\left(1-\frac{1}{(\lf^\ast)^\prime(\lk^\ast(x))}\right)+z \phi_{B_\eps}^\prime(z)=\eta\lf^\ast(\lk^\ast(x))=\eta\phi(x).$$
	It is clear that if $z\to\infty$ then $x=x(z)\to\infty$ and $t_0\to0$. 
	Moreover, since $\phi{=\lf^\ast\circ\lk^\ast}\in\mathcal{R}_{\beta r^\ast}$, we have
	$$z \phi_{B_\eps}'(z)\sim\beta r^\ast \phi_{B_{\eps}}(z)\sim \beta r^\ast B_\eps^{1-\beta r^\ast} \phi(z)$$
	and (see Lemma \ref{Lemt})
	$$\overline{s}^\prime\left(1-\frac{1}{(\lf^\ast)^\prime(\lk^\ast(x))}\right) \sim r^{1/\beta}(\beta-1)^{1/\beta}\frac{\phi(x)}{x}.$$
	
	Thus
	$$ r^{1/\beta}(\beta-1)^{1/\beta}\frac{z}{x}(1+o(1)) + \beta r^\ast B_\eps^{1-\beta r^\ast} \frac{\phi(z)}{\phi(x)}(1+o(1))=\eta.$$
	
	Take arbitrary sequence $z_n\to\infty$, set $x_n=x(z_n)$ and define $y_n=\frac{z_n}{x_n}$. 
	We {have
	\begin{align}\label{nnn}
	\underline{C}_1 \frac{z_n}{x_n}+\underline{C}_2 \frac{\phi(z_n)}{\phi(x_n)}\leq \eta \leq \overline{C}_1 \frac{z_n}{x_n}+\overline{C}_2 \frac{\phi(z_n)}{\phi(x_n)}
	\end{align}
 for some positive constants $\underline{C}_i, \overline{C}_i$, $i=1,2$. Thus, by the first inequality above we quickly infer that $y_n=\frac{z_n}{x_n}\leq \eta/\underline{C}_1$. 
	By Potter bounds (\cite[Theorem~1.5.6]{BGT89}) we have that for any $A>1$ and $\delta>0$
	$$\frac{\phi(z)}{\phi(x)}\leq A \max\left\{ \left(\frac{z}{x}\right)^{\beta r^\ast+\delta},\left(\frac{z}{x}\right)^{\beta r^\ast-\delta} \right\}$$
	for sufficiently large $z$ and $x$. Hence, the second inequality in \eqref{nnn} gives us 
	$$0<\eta\leq \max\{\overline{C}_1,A\overline{C_2}\} \max\{y_n,y_n^{\beta r^\ast\pm\delta}\}$$ 
	and so} $\lambda_1\leq y_n\leq \lambda_2$ for some positive constants $\lambda_1,\lambda_2$. Thus there exists a convergent subsequence $y_{n_k}$ to $D$, say, for which we also have ($x_n=z_n/y_n$)
	$$ r^{1/\beta}(\beta-1)^{1/\beta}y_{n_k}(1+o(1)) + \beta r^\ast B_\eps^{1-\beta r^\ast} \frac{\phi(z_{n_k})}{\phi(z_{n_k}/y_{n_k})}(1+o(1))=\eta,$$
	{where $o(1)$ is with respect to $z_{n,k}\to\infty$.}
	Thanks to uniform convergence in \eqref{reg}, we see that 
	$$\frac{\phi(z_{n_k})}{\phi(z_{n_k}/y_{n_k})} \to D^{\beta r^\ast}$$ 
	and 	$D=D(B,\eta,\eps)$ satisfies
	\begin{align}\label{eqD}
	r^{1/\beta}(\beta-1)^{1/\beta}D + \beta r^\ast B_\eps^{1-\beta r^\ast} D^{\beta r^\ast}=\eta.
	\end{align}
	Since such $D$ is unique {(left hand side of \eqref{eqD} is strictly increasing in $D>0$)}, we conclude that $z\sim D x$.
	
	
Recall that we have
	$$
	K_2(z)\leq C_\eta \exp\left( z \overline{s}(1-t_0)-zt_0 \phi_{B_\eps}^\prime(z)-\eta t_0\lf\left(\frac{1}{t_0}\right)   \right)
  $$
 for some finite constant $C_\eta$. By \eqref{deft0},
\begin{align*}
z \overline{s}(1-t_0)-zt_0 \phi_{B_\eps}^\prime(z)-\eta t_0\lf\left(\frac{1}{t_0}\right) 
= z \overline{s}(1-t_0)+z t_0 \overline{s}^\prime(1-t_0)-\eta \lf^\prime\left(\frac{1}{t_0}\right).
\end{align*}
By Lemma \ref{Lemt} we have
\begin{align*}
z \overline{s}(1-t_0)&\sim\alpha(\beta-1)^{1/\beta} r^{-1/\alpha} D \lk^\ast(x)
\intertext{and}
z t_0 \overline{s}^\prime(1-t_0) &\sim D x \frac{1}{(\lf^\ast)^\prime(\lk^\ast(x))}\frac{r^{1/\beta}(\beta-1)^{1/\beta}\lf^\ast(\lk^\ast(x))}{x}\\
&\sim r^{1/\beta}(\beta-1)^{1/\beta}(r^\ast)^{-1} D k^\ast(x).
\end{align*}
Thus, 
\begin{align}\label{logK}
\limsup_{z\to\infty}\frac{\log K_2(z)}{k^\ast(x)}\leq 	\alpha(\beta-1)^{1/\beta} r^{-1/\alpha} D+ r^{1/\beta}(\beta-1)^{1/\beta}(r^\ast)^{-1}D-\eta.
\end{align}
If the right hand side above is negative, then {for some $\zeta>0$ we have} $K_2(z)\leq{\exp(-\zeta k^\ast(x))} \to 0$ as $z\to\infty$ and the same holds for $I_B$.
We will show that if $B<B_{co}$, then the right hand side of \eqref{logK} is negative for some $\eta,\eps\in(0,1)$.
Right hand side of \eqref{logK} is negative if
$$D<\eta\frac{r^{1/\alpha}}{(\alpha+r-1)(\beta-1)^{1/\beta}}=:\overline{D},$$
where we have used
$$\alpha r^{-1/\alpha}+r^{1/\beta}(r^\ast)^{-1}=r^{-1/\alpha}(\alpha+r-1).$$
We will show that for fixed $\eta$ and $\eps$, function $B\mapsto D(B,\eta,\eps)$ is strictly increasing. 
Let $0<B_1<B_2$ and put $D_i=D(B_i,\eta,\eps)$, $i=1,2$. Then by \eqref{eqD} we obtain (recall that $1-\beta r^\ast<0$)
\begin{align*}
0&=r^{1/\beta}(\beta-1)^{1/\beta}(D_1-D_2) + \frac{\beta r^\ast}{(1-\eps)^{1-\beta r^\ast}} \left( B_1^{1-\beta r^\ast} D_1^{\beta r^\ast}-B_2^{1-\beta r^\ast} D_2^{\beta r^\ast}\right) \\
&>r^{1/\beta}(\beta-1)^{1/\beta}(D_1-D_2) + \frac{\beta r^\ast}{(1-\eps)^{1-\beta r^\ast}}B_2^{1-\beta r^\ast} \left( D_1^{\beta r^\ast}- D_2^{\beta r^\ast}\right),
\end{align*}
which implies that $D_2>D_1$. 
Moreover, after tedious but straightforward calculations one can show that
for 
$$\overline{B}:=\eta(1-\eps)\frac{\alpha-1}{\alpha+r-1}\left(\frac{r}{\alpha-1}\right)^{r/(\alpha+r-1)}  \left(\frac{\gamma}{\gamma-1}\right)^{\gamma-1}=\eta(1-\eps)B_{co}$$
one has $D(\overline{B},\eta,\eps)=\overline{D}$. {To see this, insert definition of $\overline{D}$ into \eqref{eqD} and calculate $B$. It is equal to $\overline{B}$.}

Thus, for any $B<B_{co}$, there exists $\eta, \eps\in(0,1)$ such that $B<\overline{B}$ and thus $D(B,\eta,\eps)<\overline{D}$.
\end{proof}

\begin{lemma}\label{Lemt}
Under the assumptions of Theorem \ref{THco}, assume that $z$ and $t_0$ are related by \eqref{deft0}. 
Let $t_0=1/(\lf^\ast)^\prime(\lk^\ast(x))$, $\phi=\lf^\ast\circ \lk^\ast$ and function $\overline{s}$ be defined as in \eqref{defs}. Then, as $z\to\infty$, we have
\begin{itemize}
\item[a)]
$\overline{s}(1-t_0)\sim\alpha(\beta-1)^{1/\beta} r^{-1/\alpha} \frac{\lk^\ast(x)}{x}$,
\item[b)]
$\overline{s}^\prime(1-t_0) \sim r^{1/\beta}(\beta-1)^{1/\beta}\frac{\phi(x)}{x}.$
\end{itemize}
\end{lemma}
\begin{proof}
$a)$ Since $f$ is regularly varying and $\lf\sim\uf$, we have $\lf^\ast\sim \uf^\ast$.
Thus,
$$t_0 \uf\left(t_0^{-1}\right)\sim \uf^\prime\left(t_0^{-1}\right)/r\sim \lk^\ast(x)/r.$$
{Hence},
$$x \overline{s}(1-t_0)=x \lk^{-1}\left(t_0\uf(1/t_0)\right)\sim x \lk^{-1}\left(\lk^\ast(x)/r\right)\sim x r^{-1/\alpha} \lk^{-1}\left(\lk^\ast(x)\right).$$
Moreover, by Lemma \ref{kinv} {with the substitution $x\mapsto k^\ast(x)$,} we have
$$\lk^{-1}\left(\lk^\ast(x)\right) \sim \alpha(\beta-1)^{1/\beta} \lk^\ast(x)/x.$$

$b)$ We have
$$\overline{s}^\prime(1-t_0)=\frac{\uf^\prime(1/t_0)/t_0-\uf(1/t_0)}{\lk^\prime(\overline{s}(1-t_0))}.$$
 By \eqref{eqstar} the numerator above equals 
 $$\uf^\ast\left(\uf^\prime(1/t_0)\right)=\uf^\ast\left(\uf^\prime( \lf^\prime(\lk^\ast(x)))\right)\sim\phi(x).$$
 By $a)$,
 $$\overline{s}(1-t_0)\sim\alpha(\beta-1)^{1/\beta} r^{-1/\alpha}\beta^{-1} (\lk^\ast)^\prime(x)$$
 and thus
 $$\lk^\prime(\overline{s}(1-t_0))\sim \lk^\prime(\alpha(\beta-1)^{1/\beta} r^{-1/\alpha}\beta (\lk^\ast)^\prime(x))\sim (\alpha(\beta-1)^{1/\beta} r^{-1/\alpha}\beta^{-1})^{\alpha-1} x,$$
{where the latter asymptotic equivalence follows from the fact that $\lk^\prime\in\mathcal{R}_{\alpha-1}$ and $\lk^\prime\circ (\lk^\ast)^\prime=\mathrm{Id}$.}
Finally observe that since $\alpha^{-1}+\beta^{-1}=1$, we have
$$r^{1/\beta}(\beta-1)^{1/\beta}=(\alpha(\beta-1)^{1/\beta} r^{-1/\alpha}\beta^{-1})^{1-\alpha}.$$
\end{proof}

\section{Proofs {of auxiliary results}}\label{secProof}
\begin{proof}[Proof of Lemma~\ref{kinv}]
	Suppose first that $f\in \mathcal{SR}_\alpha$. Then, $f^\prime\in\mathcal{SR}_{\alpha-1}$ has inverse on some neighbourhood of infinity. Since $(f^\ast)^\prime=(f^\prime)^{-1}$ we see that $(f^\ast)^\prime\in\mathcal{SR}_{1/(\alpha-1)}$ and so $f^\ast\in \mathcal{SR}_\beta$. 
	By \eqref{eqstar} and \eqref{SR} we have
	$$\frac{f\left((f^\ast)^\prime(x)\right)}{f^\ast(x)}=\frac{x (f^\ast)^\prime(x)-f^\ast(x)}{f^\ast(x)}\to \beta-1,\quad\mbox{ as }x\to\infty.$$
	Since $(f^\prime)^{-1}=(f^\ast)^\prime$, setting above $x(z)=f^\prime\left(f^{-1}(z)\right)\to\infty$ we obtain
	$$
	\frac{z}{f^\ast(x(z))}\to \beta-1.
	$$
	Thus, e.g. \cite[Lemma 2.1]{BK17} gives us
	$$(f^\ast)^{-1}(z)\sim (\beta-1)^{1/\beta} x(z)=(\beta-1)^{1/\beta}f^\prime\left(f^{-1}(z)\right) $$
	and
	$$\frac{f^{-1}(x)(f^\ast)^{-1}(x)}{x}\sim (\beta-1)^{1/\beta}\frac{f^{-1}(x)f^\prime\left(f^{-1}(x)\right)}{x}\to (\beta-1)^{1/\beta} \alpha.$$
	by the definition of $\mathcal{SR}_\alpha$.
	
	In the general case, the Smooth Variation Theorem yields the existence of $\lf,\uf\in \mathcal{SR}_\alpha$ with $\lf\leq f\leq \uf$ on some neighbourhood of infinity. Since conjugacy is order reversing, we have $\uf^\ast\leq f^\ast\leq \lf^\ast$. Moreover, $\uf^{-1}\leq f^{\leftarrow}\leq \lf^{-1}$ on a vicinity of infinity and similar inequalities hold for $(f^\ast)^{\leftarrow}$. The conclusion follows by the fact that $\lf(x)\sim\uf(x)$.
\end{proof}

\begin{proof}[Proof of Lemma~\ref{LB}]
	We will use the fact that a stochastic recursion \eqref{it} converges in distribution to the solution of affine equation.	
	Take $R_0=R'_0=0$ a.s. We proceed by induction. Assume that for some $n\in\mathbb{N}$ one has
	\begin{align}\label{ind}
	\E f(R_n)\leq \E f(R'_{n})\qquad\mbox{for all convex functions $f$ on $\RR$}.
	\end{align}
	Let $f$ be a convex function. By the fact that $r\mapsto \E f(M r+Q)$ is convex and by the inductive assumption, we first infer that
	$$\E f(M_{n+1}R_n+Q_{n+1})\leq \E f(M_{n+1}R'_{n}+Q_{n+1}).$$
	Further, for any $r\geq0$, the function $h_r(m,q):=f(m r+ q)$ is supermodular. Note that since $R_0'=0$ and $M, Q$ are a.s. non-negative, $R_n'$ is a.s. non-negative as well. Then,
	\begin{multline*}
		\E f(R_{n+1}) \leq \E f(M_{n+1}R'_n+Q_{n+1}) = \E h_{R'_n}(M_{n+1},Q_{n+1}) 
		\leq \E h_{R'_n}(M'_{n+1},Q'_{n+1})=\E f(R'_{n+1}).
	\end{multline*}
	Thus we have established \eqref{ind} for any $n\in\mathbb{N}$. 
	Observe that $(R_{n})_n$ is stochastically non-decreasing, that is,
	$$R_{n+1}\stackrel{d}{=}\sum_{k=1}^{n+1} M_1\cdot\ldots\cdot M_{k-1}Q_k\geq\sum_{k=1}^{n} M_1\cdot\ldots\cdot M_{k-1}Q_k\stackrel{d}{=}R_n.$$
	Thus, for any {weakly monotonic} function $f$, $(f(R_n))_n$ is stochastically {weakly monotonic} as well {and the same holds for $(f(R'_n))_n$}. 
	{
	Assertion follows by the fact that $\E f(R_{n})$ and $\E f(R_{n}')$ are weakly monotone and so have a limit (possibly infinite) as $n\to\infty$.
	}
\end{proof}

\begin{proof}[Proof of Theorem~\ref{Kas}]
	In \cite[Theorem 4.12.7]{BGT89} a different formulation of the same result is proposed. Namely, if $\alpha\in(0,1)$, $\phi\in \mathcal{R}_\alpha$, define $\psi(z)=z/\phi(z)\in \mathcal{R}_{1-\alpha}$.
	Then,
	$$-\log\P(X>x)\sim \phi^{\leftarrow}(x)$$
	if and only if 
	$$\log M(z)\sim (1-\alpha)\alpha^{\alpha/(1-\alpha)}\psi^{\leftarrow}(z).$$
	
	We have to show that
	$$ k(x)\sim \phi^{\leftarrow}(x)\qquad\mbox{if and only if}\qquad k^\ast(z)\sim (1-\alpha)\alpha^{\alpha/(1-\alpha)}\psi^{\leftarrow}(z).$$
	Let $\rho=1/\alpha$ and put $f=\phi^{\leftarrow}\in\mathcal{R}_{\rho}$.
	By Lemma \ref{kinv}, we have
	\begin{align}\label{fff}
	\frac{f^{\leftarrow}(x) (f^\ast)^{\leftarrow}(x)}{x}\to  \rho(\rho-1)^{-(\rho-1)/\rho}= \alpha^{-\alpha} (1-\alpha)^{-(1-\alpha)},\qquad\mbox{ as }x\to\infty.
	\end{align}
	But 
	$$\frac{f^{\leftarrow}(x) (f^\ast)^{\leftarrow}(x)}{x}\sim\frac{(f^\ast)^{\leftarrow}(x)}{\psi(x)}$$
	and so \eqref{fff} is equivalent to {(use definition of asymptotic inverse and Lemma 2.1 in \cite{BK17})}
	$$f^\ast(z)\sim (1-\alpha)\alpha^{\alpha/(1-\alpha)}\psi^{\leftarrow}(z).$$
\end{proof}

\begin{proof}[Proof of Theorem~\ref{deB}]
	Recall that $h\in \mathcal{R}_\rho(0+)$ if $x\mapsto h(1/x)\in\mathcal{R}_{-\rho}$. {Moreover, if $h\in \mathcal{R}_\rho(0+)$, then $h^\leftarrow\in\mathcal{R}_{1/\rho}$. 
	Indeed, we have $h(1/x)=x^{-\rho} L(x)$ for some slowly varying function $L$. The (asymptotic) inverse $g$ of $x\mapsto x^{-\rho} L(x)$ is regularly varying with index $-1/\rho$. But then we have $h^\leftarrow(x)\sim1/g(x)$.
	}
	
	In \cite[Theorem 4.12.9]{BGT89} the following result is proved: for $\alpha<0$ and $\phi\in \mathcal{R}_\alpha(0+)$, define $\psi(z)=\phi(z)/z\in \mathcal{R}_{\alpha-1}(0+)$.
	Then,
	$$-\log\P(Y\leq x)\sim 1/\phi^{\leftarrow}(1/x)\qquad(x\to0+)$$
	if and only if 
	$$-\log \E e^{-\lambda Y}\sim (1-\alpha)(-\alpha)^{\alpha/(1-\alpha)}/\psi^{\leftarrow}(\lambda)\qquad(\lambda\to\infty).$$

	First observe that, under regular variation, asymptotics of $-\log\P(Y\leq 1/x)$ and $-\log\P(Y< 1/x)$ are the same. Indeed, for any $\eps>0$, we have $-\log\P(Y\leq 1/(x+\eps))\sim-\log\P(Y\leq 1/(x-\eps))$.
	Further, it is easy to see that if $f$ is regularly varying, then
  $$f(x)\sim x/\phi^\leftarrow(x)\quad \mbox{ if and only if }\quad f^\leftarrow(x)\sim x\psi^\leftarrow(x).$$
  It is left to show that 
	$$f(x)\sim x/\phi^\leftarrow(x)\quad \mbox{ if and only if }\quad (f^\ast)^{{\leftarrow}}(\lambda)\sim (1-\alpha)(-\alpha)^{\alpha/(1-\alpha)}/\psi^{\leftarrow}(\lambda).$$
	Since $x\mapsto \phi(1/x)\in\mathcal{R}_{-\alpha}$, we see that {$\phi^\leftarrow\in\mathcal{R}_{1/\alpha}$ and so }$f\in \mathcal{R}_{\rho}$ with $\rho=1-\alpha^{-1}$. 
	By Lemma \ref{kinv} we have
	$$\frac{f^{\leftarrow}(x) (f^\ast)^{\leftarrow}(x)}{x}\to (1-\alpha)(-\alpha)^{\alpha/(1-\alpha)},\qquad\mbox{ as }x\to\infty,$$
which ends the proof.
\end{proof}

\begin{proof}[Proof of Theorem~\ref{propH}]
\begin{itemize}
\item[a)] By the definition of $h$, for any $x$ and any positive number $g(x)$, there exists a number $t(x)$ such that
$$h(x)\leq -t(x)\log\P\left(\frac{1}{1-M}>t(x), Q>\frac{x}{t(x)}\right)\leq h(x)+g(x).$$
If $g(x)=o(1)$, we obtain the first part of the assertion. 

Using the fact that 
$$\P\left(\frac{1}{1-M}>t(x), Q>\frac{x}{t(x)}\right)\leq \P\left(\frac{Q}{1-M}>x\right),$$
we obtain \eqref{hineq}.

\item[b)] The assertion follows from the Fr\'{e}chet–-Hoeffding bounds
\begin{align*}
\P\left(\frac{1}{1-M}>t\right)&+\P\left(Q>\frac{x}{t}\right)-1\\
&\leq\P\left(\frac{1}{1-M}>t, Q>\frac{x}{t}\right)\leq
 \min\left\{\P\left(\frac{1}{1-M}>t\right), \P\left(Q>\frac{x}{t}\right)\right\}.
\end{align*}

\item[c)]
First part follows quickly by Theorem \ref{HU}; see Remark \ref{hfk}. Moreover, we already know that if $f\in\mathcal{R}_r$ and $k\in\mathcal{R}_\alpha$ with $r,\alpha>1$, then 
$(f^\ast\circ k^\ast)^\ast\in\mathcal{R}_\gamma$. 
Thanks to Smooth Variation Theorem, we may only consider the case when $f\in\mathcal{SR}_{r}$ and $k\in\mathcal{SR}_\alpha$. The infimum in the definition of 
$h_{ind}(x)$
is attained at a point $t_1=t_1(x)$ such that
\begin{align}\label{ggg}
f^\prime(t_1)=\frac{x}{t_{{1}}}k^\prime\left(\frac{x}{t_1}\right)-k\left(\frac{x}{t_1}\right)=k^\ast\left(k^\prime\left(\frac{x}{t_1}\right)\right),
\end{align}
where the last equality is \eqref{eqstar}. 
Thus, by regular variation of $f$ and $k$, we obtain
$$x=t_1\left[(k^\ast)^\prime\circ(k^\ast)^{-1}\circ f^\prime\right](t_1)=t_1^{(r+\alpha-1)/\alpha}L(t_1).$$
for some slowly varying function $L$. This means that $t_1\to\infty$ and $x/t_1\to\infty$ as $x\to\infty$ and that $x\mapsto t_1(x)\in\mathcal{R}_{\alpha/(r+\alpha-1)}$.

Consider now the case when $\mathrm{ess}\sup Q=q_+<\infty$. In such case $k(x)=\infty$ if $x\geq q_+$ and so
$$k^\ast(z)=\sup_{x>0}\{zx-k(x)\}=\sup_{x<q_+}\{zx-k(x)\}\leq\sup_{x<q_+}\{zx\}=z q_+.$$
On the other hand, for any $x<q_+$,
$$\frac{k^\ast(z)}{z}\geq x-\frac{k(x)}{z}\to x\qquad\mbox{ as }z\to\infty.$$
Then, we have
\begin{align*}
(f^\ast\circ k^\ast)^\ast(x)&=\sup_{z>0}\left\{zx- f^\ast\left(k^\ast(z)\right)\right\}
\\
&\sim \sup_{z>0}\left\{zx-f^\ast(q_+ z)\right\} 
= \sup_{y>0}\left\{\frac{x}{q_+}y-f^\ast(y)\right\}=f\left(\frac{x}{q_+}\right).
\end{align*}

\item[d)] As previously, we work with $f\in\mathcal{SR}_{r}$ and $k\in\mathcal{SR}_\alpha$. 
{Let us first make a simple observation that if functions $a$ and $b$ are continuous, $a(x_0)<b(x_0)$ and $a(x_1)>b(x_1)$, $a$ is increasing and $b$ is decreasing, then there exists a unique $t_0$ such that $a(t_0)=b(t_0)$ and moreover
$\inf_{t\in[x_0,x_1]}\max\{ a(t),b(t)\}=a(t_0)$. 
Our first step here will be to show that the infimum in the definition of 
$$h_{co}=\inf_{t\geq 1}\left\{\max\left\{f(t),t\,k\left(\frac{x}{t}\right)\right\}\right\}$$ 
is (for $x$ large enough) attained at a point $t_2=t_2(x)$ such that $f(t_2)=t_2 k(x/t_2)$. 
In our case, function $[1,\infty)\ni t\mapsto t\,k\left(\frac{x}{t}\right)$ may not be decreasing and so we can't use our observation directly. However,
note that since $k\in\mathcal{SR}_\alpha$, the limit 
$$\lim_{z\to\infty}\frac{z k^\prime(z)}{k(z)}=\alpha$$
is strictly larger than $1$ and so $k(z)<z k^\prime(z)$ for $z$ large enough, say $z\geq1/T$ for some $T>0$. 
Calculating the derivative of $t\mapsto t\, k(x/t)$ we obtain
$$k\left(\frac{x}{t}\right)-\frac{x}{t}k^\prime\left(\frac{x}{t}\right)$$
which is strictly negative if $x/t\geq 1/T$, that is, $t\leq T x$.
Hence, for $t\in [1,T x)$, $f(t)$ is increasing and $t\, k(x/t)$ is decreasing. Moreover, for large $x$ we have $f(1)<k(x)$ and $f(x)>x k(1)$ and so 
$$\inf_{t\in [1,T x]} \left\{\max\left\{f(t),t\,k\left(\frac{x}{t}\right)\right\}\right\}=f(t_2).$$
It is enough to show that for $x$ large enough, the infimum in the definition of $h_{co}$ is not attained on the set $(T x,\infty)$. 
We have 
$$\inf_{t> T x} \left\{\max\left\{f(t),t\,k\left(\frac{x}{t}\right)\right\}\right\}\geq \max\{ f(T x), T x\sup_{z\in(0,1/T)} k(z)\}\sim T^r f(x) \in \mathcal{R}_{r}.$$
Let us assume for a while that 
\begin{align}\label{TOSHOW}
x\mapsto f(t_2(x))\in\mathcal{R}_{\alpha r/(\alpha+r-1)}.
\end{align} 
Since, $r>\alpha r/(\alpha+r-1)$ we see that under \eqref{TOSHOW} our claim is true and we have 
$$h_{co}(x)=\inf_{t\in [1,T x]} \left\{\max\left\{f(t),t\,k\left(\frac{x}{t}\right)\right\}\right\}=f(t_2(x)).$$
}
{But $t\mapsto t\, k^{-1}(f(t)/t)\in \mathcal{R}_{(r^\ast+\alpha-1)/\alpha}$ and $x=t_2\, k^{-1}(f(t_2)/t_2)$, so we have established \eqref{TOSHOW}. Moreover,} it is easy to see that, as before, $t_2$ and $x/t_2$ go to infinity as $x\to\infty$. Thus,
$$k\left(\frac{x}{t_2}\right)=\frac{f(t_2)}{t_2}\sim \frac1r f'(t_2).$$
{Hence,}
$$x\sim \frac{t_2 g(t_2)}{r^{1/\alpha}}.$$
where $g=k^{-1}\circ f^\prime\in\mathcal{R}_{(r-1)/\alpha}$. {In this way we have established \eqref{TOSHOW}.}

On the other hand, in the case of independent $M$ and $Q$, \eqref{ggg} implies that 
\begin{align}\label{og}
f^\prime(t_1)\sim (\alpha-1) k\left(\frac{x}{t_1}\right)
\end{align}
and so
$$x\sim \frac{t_1 g(t_1)}{(\alpha-1)^{1/\alpha}}.$$
Thus{, $r^{-1/\alpha} t_2 g(t_2)\sim (\alpha-1)^{-1/\alpha} t_1 g(t_1)$, $t\mapsto t g(t)\in\mathcal{R}_{(r+\alpha-1)/\alpha}$ and so by Lemma 2.1 in \cite{BK17} we obtain}
$$t_2\sim t_1\left(\frac{r}{\alpha-1}\right)^{1/(\alpha+r-1)}.$$
Finally
$$h_{co}(x)= f(t_2)\sim \left(\frac{r}{\alpha-1}\right)^{r/(\alpha+r-1)} f(t_1)\sim \frac{\alpha-1}{\alpha+r-1}\left(\frac{r}{\alpha-1}\right)^{r/(\alpha+r-1)}h_{ind}(x),$$
since
$$h_{ind}(x)=f(t_1)+t_1 k(x/t_1)\sim f(t_1)+(\alpha-1)^{-1}t_1 f^\prime(t_1)\sim \left(1+\frac{r}{\alpha-1}\right) f(t_1).$$

\item[e)]
The infimum in the definition of $h_{counter}(x)$ is calculated for $t>0$ such that
\begin{align}\label{contr}
\P\left(M>1-\frac{1}{t}\right)+\P\left(Q>\frac{x}{t}\right)>1.
\end{align}
We will show that, as $x\to\infty$, the infimum is actually calculated for 
$$t\in I_x:=[1,(1-m_-)^{-1}]\cup[x/q_-,\infty).$$ Take $t\in[1,\infty)\setminus I_x$. Then, there exist $m\in(m_-,m_+)$ and $q>q_-$  with $t\in(1/(1-m),x/q)$. 
We have $\P(M>m)<1$ and $\P(Q>q)<1$.
Consider first the case of $t\in\left(1/(1-m),\sqrt{x}\right]$. Then
$$\P\left(M>1-\frac{1}{t}\right)+\P\left(Q>\frac{x}{t}\right)\leq\P\left(M>m\right)+\P\left(Q>\sqrt{x}\right)$$
and we obtain a contradiction with \eqref{contr} as $x\to\infty$.
Similarly, if $t\in[\sqrt{x},x/q)$, then we obtain
$$\P\left(M>1-\frac{1}{\sqrt{x}}\right)+\P\left(Q>q\right)>1$$
and this yields a contradiction as well if $x\to\infty$. 

So far, we have shown that
\begin{align*}
h_{counter}(x)& \sim \inf_{t\in I_x}\left\{ -t\log\left[\P\left(M>1-\frac1t\right)+\P\left(Q>\frac xt\right)-1\right]\right\}\\
&= \min\left\{ \inf_{t\in[1,1/(1-m_-)]}\left\{ \cdots \right\}, \inf_{t\geq x/q_-}\left\{ \cdots \right\} \right\}.
\end{align*}
If $\P(M=m_-)=0=\P(Q=q_-)$, this is exactly \eqref{hcount} since ($f$ and $k$ are right-continuous and non-decreasing)
\begin{align*}
\inf_{t\geq x/q_-}\left\{ \cdots \right\}=\inf_{t\geq x/q_-} f(t)=f(\frac{x}{q_-})
\intertext{and similarly}
\inf_{t\in[1,1/(1-m_-)]}\left\{ \cdots \right\} =\inf_{t\leq1/(1-m_-)} t k\left(\frac xt\right)=\frac{k((1-m_-)x)}{1-m_-}.
\end{align*}
On the other hand, if $\P(Q=q_-)>0$, then by \eqref{contr}, we see that $t=x/q_-$ is impossible as $x\to\infty$ and thus the infimum is calculated for $t\in(x/q_-,\infty)$. However, this introduces virtually no changes {to the proof} since $\inf_{1\leq t<x/q_-} f(t)\sim f(x/q_-)$.
If $\P(M=m_-)>0$, then $t=1/(1-m_-)$ is impossible and we eventually obtain \eqref{hcount}.
\end{itemize}
\end{proof}

\begin{proof}[Proof of Lemma~\ref{lem2}]
	Assume that $I_{\phi}(z)\leq 1$ for $z> N$ and define 
	$$\overline{\phi}(x)=\begin{cases} ax & x\leq N \\ \phi(x)+C, & x> N.   \end{cases}$$
	We will show that for sufficiently large $a$ and $C$, 
	\begin{align}\label{ovphi}
	I_{\overline{\phi}}(z)\leq 1\qquad\mbox{ for all }z\geq 0.
	\end{align}
	
	Observe that $I_{\overline{\phi}}(0)=1$. If $a>\E Q(1-\E M)^{-1}$, then $I_{\overline{\phi}}^\prime(0)=\E Q-a+a\E M<0$, thus, there exists $\varepsilon>0$ such that $I_{\overline{\phi}}(z)\leq 1$ for $z\in[0,\varepsilon)$. 
	For $z\in[\varepsilon,N]$ we have
	$$I_{\overline{\phi}}(z)\leq \E e^{N Q-a \varepsilon(1-M)}$$
	and the right hand side tends to $0$ as $a\to\infty$. Thus, for sufficiently large $a$ we also have $I_{\overline{\phi}}(z)\leq 1$ for $z\in[\varepsilon,N]$.
	Further, for $z>N$, we have
	$$I_{\overline{\phi}}(z)=I_{\phi}(z)+\E e^{z Q}\left(e^{a z M-\phi(z)-C}-e^{\phi(zM)-\phi(z)}\right) I_{zM\leq N}$$
	and one may find $C$ such that $ax-C\leq \phi(x)$ for any $x\in[0,N]$, so the second term above is non-positive. 
	
	Further, proceeding by induction, assume that $\E \exp(z R_n){\leq} \exp(\overline{\phi}(z))$ for all $z\geq 0$ and some $n\in\mathbb{N}$.
	Then, for $z\geq0$,
	$$\E \exp(z R_{n+1})=\E \exp(z M_{n+1}R_{n}+z Q_{n+1})\leq \E \exp(\overline{\phi}(zM)+z Q)\leq \exp(\overline{\phi}(z))$$
	by \eqref{ovphi}.
	Moreover, we can start the induction since $R_0$ can be chosen arbitrary and thus, passing to the limit as $n\to\infty$, we obtain the assertion.
\end{proof}

\section*{Acknowledgements} The author would like to express his sincere gratitude to the referees for their useful comments. The author thanks to prof. Iksanov for his stimulating remarks that led to an improvement of the work. The author was partially supported by the NCN Grant UMO-2015/19/D/ST1/03107. 

\bibliographystyle{plainnat}

\bibliography{BiblClean}

\end{document}